\documentclass[onefignum,onetabnum]{siamart171218}

\usepackage{amssymb}
\usepackage{bm}
\usepackage{upgreek}
\usepackage{amsmath}
\usepackage{comment}
\usepackage{amsfonts}
\usepackage{graphicx}
\usepackage{epstopdf}
\usepackage{algorithmic}
\usepackage{textcomp}
\usepackage{listings}
\usepackage{hyperref}
\usepackage{subcaption}
\usepackage{tikz}
\usepackage[utf8]{inputenc}
\usepackage{pgfplots}
\usepackage{pbox}

\pgfplotsset{every tick label/.append style={font=\footnotesize}}
\pgfplotsset{label style={font=\footnotesize}}

\graphicspath{{img/all/}}

\usepackage{amsopn}

\usepackage{cleveref}

\ifpdf
  \DeclareGraphicsExtensions{.eps,.pdf,.png,.jpg}
\else
  \DeclareGraphicsExtensions{.eps}
\fi

\newsiamremark{remark}{Remark}
\newsiamremark{hypothesis}{Hypothesis}
\crefname{hypothesis}{Hypothesis}{Hypotheses}
\newsiamthm{claim}{Claim}

\newsiamremark{assumption}{Assumption}
\crefname{assumption}{assumption}{assumptions}

\definecolor{Dcolor}{rgb}{1,.2,.2}
\definecolor{NDcolor}{rgb}{.2,.2,1}
\definecolor{Rcolor}{rgb}{.2,1,.2}
\definecolor{usualcolor}{rgb}{.6,.6,.6}
\definecolor{reference}{rgb}{0,0,0}

\newcommand{\Dcolordesc}{red}
\newcommand{{\NDcolordesc}}{blue}
\newcommand{\Rcolordesc}{green}
\newcommand{\usualcolordesc}{grey}
\newcommand{\usualvariantcolordesc}{grey dashed}

\newcommand{\linlindesc}{circles}
\newcommand{\conlindesc}{squares}
\newcommand{\onedesc}{triangles}
\newcommand{\twodesc}{diamonds}
\newcommand{\threedesc}{pentagons}

\definecolor{red}{rgb}{1,0,0}

\newcommand{\Dlinlinplot  }[1][2]{\addplot [Dcolor, mark=*, thick, mark size=#1, mark options={solid,draw=black}]}

\newcommand{\Doneplot  }[1][3]{\addplot [Dcolor, mark=triangle*, thick, mark size=#1, mark options={solid,draw=black}]}
\newcommand{\Dtwoplot  }[1][3]{\addplot [Dcolor, mark=diamond*, thick, mark size=#1, mark options={solid,draw=black}]}
\newcommand{\Dthreeplot}[1][3]{\addplot [Dcolor, mark=pentagon*, thick, mark size=#1, mark options={solid,draw=black}]}

\newcommand{\Dlinlindesc}{{\Dcolordesc} \linlindesc}

\newcommand{\Donedesc}{{\Dcolordesc} \onedesc}
\newcommand{\Dtwodesc}{{\Dcolordesc} \twodesc}
\newcommand{\Dthreedesc}{{\Dcolordesc} \threedesc}

\newcommand{\NDlinlinplot }[1][2]{\addplot [NDcolor, mark=*,   thick, mark size=#1, mark options={solid,draw=black}]}
\newcommand{\NDconlinplot }[1][2]{\addplot [NDcolor, mark=square*,   thick, mark size=#1, mark options={solid,draw=black}]}

\newcommand{\NDoneplot  }[1][3]{\addplot [NDcolor, mark=triangle*, thick, mark size=#1, mark options={solid,draw=black}]}
\newcommand{\NDtwoplot  }[1][3]{\addplot [NDcolor, mark=diamond*, thick, mark size=#1, mark options={solid,draw=black}]}
\newcommand{\NDthreeplot}[1][3]{\addplot [NDcolor, mark=pentagon*, thick, mark size=#1, mark options={solid,draw=black}]}

\newcommand{\NDlinlindesc}{{\NDcolordesc} \linlindesc}
\newcommand{\NDconlindesc}{{\NDcolordesc} \conlindesc}

\newcommand{\NDonedesc}{{\NDcolordesc} \onedesc}
\newcommand{\NDtwodesc}{{\NDcolordesc} \twodesc}
\newcommand{\NDthreedesc}{{\NDcolordesc} \threedesc}

\newcommand{\RthreeD }{\addplot3[surf,Rcolor,z buffer=reverse y seq,faceted color=reference]}

\newcommand{\Roneplot  }[1][3]{\addplot [Rcolor, mark=triangle*, thick, mark size=#1, mark options={solid,draw=black}]}
\newcommand{\Rtwoplot  }[1][3]{\addplot [Rcolor, mark=diamond*, thick, mark size=#1, mark options={solid,draw=black}]}
\newcommand{\Rthreeplot}[1][3]{\addplot [Rcolor, mark=pentagon*, thick, mark size=#1, mark options={solid,draw=black}]}

\newcommand{\Ronedesc}{{\Rcolordesc} \onedesc}
\newcommand{\Rtwodesc}{{\Rcolordesc} \twodesc}
\newcommand{\Rthreedesc}{{\Rcolordesc} \threedesc}

\newcommand{\usualoneplot}[1][2]{\addplot [usualcolor, mark=triangle*, thick, mark size=#1, mark options={solid,draw=black}]}
\newcommand{\usualtwoplot}[1][2]{\addplot [usualcolor, mark=diamond*, thick, mark size=#1, mark options={solid,draw=black}]}

\newcommand{\usualonedesc}{{\usualcolordesc} \onedesc}
\newcommand{\usualtwodesc}{{\usualcolordesc} \twodesc}

\newcommand{\usualvariantoneplot}[1][2]{\addplot [usualcolor, dashed, mark=triangle*, thick, mark size=#1, mark options={solid,draw=black}]}
\newcommand{\usualvarianttwoplot}[1][2]{\addplot [usualcolor, dashed, mark=diamond*, thick, mark size=#1, mark options={solid,draw=black}]}
\newcommand{\usualvariantthreeplot}[1][2]{\addplot [usualcolor, dashed, mark=pentagon*, thick, mark size=#1, mark options={solid,draw=black}]}

\newcommand{\referenceplot}{\addplot [reference, dashed, very thick]}

\newcommand{\B}[1]{{\bm#1}}

\newcommand{\scurl}[1][]{\textup{curl}\ifthenelse{\equal{#1}{}}{}{_#1}}
\newcommand{\vcurl}[1][]{\textup{\textbf{curl}}\ifthenelse{\equal{#1}{}}{}{_#1}}
\newcommand{\sdiv}[1][]{\textup{div}\ifthenelse{\equal{#1}{}}{}{_#1}}
\newcommand{\vdiv}[1][]{\textup{\textbf{div}}\ifthenelse{\equal{#1}{}}{}{_#1}}
\newcommand{\pop}[1]{\mathcal{#1}}
\newcommand{\bop}[1]{\mathsf{#1}}
\newcommand{\popV}{\pop{V}}
\newcommand{\popK}{\pop{K}}
\newcommand{\bopV}{\bop{V}}
\newcommand{\bopK}{\bop{K}}
\newcommand{\bopKadj}{{\bop{K}'}}
\newcommand{\bopW}{\bop{W}}
\newcommand{\bopI}{\bop{Id}}

\newcommand{\trace}{\gamma}
\newcommand{\betaparam}{\beta}
\newcommand{\average}[1]{\left\{#1\right\}_\Gamma}

\newcommand{\dx}[1][x]{\,\mathrm{d}#1}
\newcommand{\RR}{\mathbb{R}}
\newcommand{\form}[1]{\mathcal{#1}}

\newcommand{\productspace}[1]{\mathbb{#1}}
\newcommand{\Th}{\mathcal{T}_h}

\renewcommand{\vec}[1]{\bm{#1}}
\newcommand{\LL}{L}
\newcommand{\Ltwo}{\LL^2}
\newcommand{\Vee}{\mathrm{V}}

\renewcommand{\emptyset}{\varnothing}

\newcommand{\norm}[1]{\| #1 \|}
\newcommand{\starnorm}[2][]{\norm{#2}_{*#1}}
\newcommand{\Vnorm}[2][]{\norm{#2}_{\productspace{V}#1}}
\newcommand{\Bnorm}[2][]{\norm{#2}_{\form{B}#1}}
\newcommand{\Ltwonorm}[2][\Gamma]{\norm{#2}_{\Ltwo(#1)}}
\newcommand{\Hnorm}[3][\Gamma]{\norm{#3}_{H^{#2}(#1)}}
\newcommand{\Hstarnorm}[3][\Gamma]{\seminorm{#3}_{H_*^{#2}(#1)}}

\newcommand{\seminorm}[1]{| #1 |}
\newcommand{\Hseminorm}[3][\Gamma]{\seminorm{#3}_{H^{#2}(#1)}}

\newcommand\x{\B{x}}
\newcommand\y{\B{y}}

\newcommand\D{_\textup{D}}
\newcommand\N{_\textup{N}}
\newcommand\ND{_\textup{ND}}

\newcommand\R{_\textup{R}}

\newcommand{\param}{\varepsilon}
\newcommand{\betamin}{\betaparam_{\min}}
\newcommand{\robindenom}{\omega}
\newcommand{\zerostar}[1]{\overset{*}{#1}\null}

\newcommand{\TODO}[1][]{{\color{red}TODO\ifthenelse{\equal{#1}{}}{}{: #1}}}

\headers{BEM with weakly imposed boundary conditions}{T. Betcke, E. Burman, and M. W. Scroggs}

\title{Boundary element methods with weakly imposed boundary conditions.
\thanks{Submitted to the editors June 2018.}}

\author{Timo Betcke\thanks{Department of Mathematics, University College London, WC1E 6BT, UK
  (\email{t.betcke@ucl.ac.uk}).}
\and Erik Burman\footnotemark[2]\thanks{Department of Mathematics, University College London, WC1E 6BT, UK
  (\email{e.burman@ucl.ac.uk}).}
\and Matthew W. Scroggs\thanks{Department of Mathematics, University College London, WC1E 6BT, UK
  (\email{matthew.scroggs.14@ucl.ac.uk}, \url{http://www.mscroggs.co.uk}).}}

\ifpdf
\hypersetup{
  pdftitle={Weak imposition of boundary conditions using a penalty method for boundary element methods}
  pdfauthor={T. Betcke, E. Burman, and M. W. Scroggs}
}
\fi

\begin{document}

\maketitle

\begin{abstract}
We consider boundary element methods where the Calder\'on projector is
used for the system matrix and boundary conditions are weakly imposed using a
particular variational boundary operator designed using techniques
from augmented Lagrangian methods. Regardless of the boundary conditions,
both the primal trace variable and the flux are approximated. We focus on the imposition of Dirichlet,
mixed Dirichlet--Neumann, and Robin conditions. A salient feature of the
Robin condition is that the conditioning of the system is robust also
for stiff boundary conditions. The theory is illustrated by a series of numerical examples.
\end{abstract}



\section{Introduction}
Weak imposition of boundary conditions has been very successful in the
context of finite element methods. In particular, Nitsche's method \cite{Nit71} has
recently received increased interest in the scientific computation
community. 
Our aim in this paper is to discuss how the idea behind this type of
method can be applied in the context of boundary element methods to
impose different types of boundary condition in a unified
framework.

Weak imposition of boundary conditions here means that neither the Dirichlet trace nor the Neumann trace is imposed exactly,
instead an $h$-dependent boundary condition is imposed that is weighted in such a way that optimal error estimates may be derived
and the exact boundary condition is recovered in the asymptotic limit.
Methods based on Nitsche's method have been succesfully utilised for boundary element method domain decomposition problems,
where they have been used to impose interface conditions at 1D interfaces between segments of 2D screens embedded in 3D space \cite{Gatica2009,Chouly2012}.
Our approach instead focusses on imposing
boundary conditions on the 2D boundary of a single domain problem through the addition of penalty terms to a general formulation written in terms of the multitrace
operator, in a similar vein to the method discussed in \cite{Babuska1973} for the finite element method.

The use of systems of boundary integral equations for problems with mixed boundary conditions is quite classical
\cite{Costabel1985,Stephan1987,Petersd1989,Petersd1990}. While these papers require the assembly of boundary operators on subsets of the boundary
mesh, the penalty method proposed in this paper requires only the addition of sparse mass matrices to the multitrace operator assembled on the
entire mesh. In addition to the greater simplicity of the resulting formulation, this method has the advantage that the sparse penalty terms
only affect the entries in the matrix for near interactions: this gives the resulting system a structure that can be utilised when
designing effective preconditioners.

This approach may not be competitive in the simple
case of pure Dirichlet or Neumann conditions due to the increase in the number of unknowns.
Therefore the main focus of this work is on more complex situations.
We will discuss the following four model cases:
\begin{enumerate}
\item non-homogeneous Dirichlet conditions,
\item non-homogeneous Neumann conditions,
\item mixed Dirichlet--Neumann boundary conditions,
\item generalised Robin conditions.
\end{enumerate}

We consider the Laplace equation: Find $u$ such that
\begin{subequations}\label{eq:Laplace}
\begin{align}
-\Delta u &=0       &&\text{in } \Omega,\\
u &= g\D            &&\text{on } \Gamma\D, \\
\frac{\partial u}{\partial\B\nu} &= g\N
                    &&\text{on } \Gamma\N,\\
\frac{\partial u}{\partial\B\nu} &= \frac{1}{\param} (g\D-u) + g\N
                    &&\text{on }\Gamma\R.\label{eq:LaplaceRobin} 
\end{align}\end{subequations}
Here $\Omega \subset \RR^3$ denotes a polyhedral domain with outward pointing normal
$\B\nu$ and boundary $\Gamma:= \Gamma\D \cup \Gamma\N \cup \Gamma\R$.
We assume for simplicity that the boundaries between $\Gamma\D$, $\Gamma\N$ and $\Gamma\R$ coincide with edges between the faces of $\Gamma$.
Whenever it is ambiguous, we will write $\B\nu_\x$ for the outward pointing normal at the point $\x$.
We assume that $g\D \in H^{1/2}(\Gamma\D \cup \Gamma\R)$ and $g\N \in
L^2(\Gamma\N \cup \Gamma\R)$. Observe that, by the Lax--Milgram lemma, there exists a unique solution
to \cref{eq:Laplace}. We assume that $u \in H^{3/2+\epsilon}(\Omega)$ for some $\epsilon>0$.


For the Robin boundary condition, we will use the ideas of Juntunen and
Stenberg \cite{JuSt09}. A salient feature of this type of imposition
of the Robin condition is that it is robust under singular
perturbations. Indeed regardless of the Robin coefficient, the
conditioning of the resulting system matrix is no worse than
for the Neumann or the Dirichlet problem.


The proposed framework is flexible and allows for the design of a range of
different methods depending on the choice of weights and residuals. We
will present a sample of possible methods with the ambition of showing
the versatility of the framework rather than claiming
that for each case the choices are optimal.

An outline of the paper is as follows. First, we review some of the
basic elements of the theory of boundary operators in \cref{sec:b_op}. Then, in \cref{sec:weak} we discuss
the design of formulations for the linear model problems in a formal setting. We propose the
corresponding boundary element methods in \cref{sec:bem} and give an
abstract analysis. The boundary elements obtained using the
formulations from \cref{sec:weak} are then shown to satisfy the
assumptions of the abstract theory.
Finally, we show some computational examples in \cref{sec:numerical}.

While the present paper focuses on weak imposition of boundary conditions through Nitsche type coupling for BEM, ultimately the goal is to 
develop a framework for complex BEM/BEM and FEM/BEM multiphysics coupling situations. Existing approaches here are often built upon FETI 
and BETI type methods \cite{Langer2003, Langer2005}. While BETI is usually formulated in terms of Steklov--Poincar\'e operators, the 
framework proposed in this paper builds directly upon Calder\'on projectors of the subdomains.

For the method proposed in the present work the multi-domain coupling will take a form similar to that using Nitsche's method in the FEM/FEM 
coupling setting of \cite{Becker03}; see also the FEM/BEM coupling of \cite{ChernovUNP} where a Nitsche's method for the coupling was 
proposed, using the Steklov-Poincar\'e operator for the BEM system.

An important application area for the presented weak imposition of boundary conditions are inverse problems with unknown boundary 
conditions. Since the boundary condition only enters through a sparse operator this can be easily updated in each step of a solver 
iteration, while the boundary integral operators only need to be computed once.
In particular, for reconstruction of the coefficient in a Robin condition (see eg \cite{Jin10} for a finite element approach and \cite{Bara16}
for a detailed analysis of the stability of this problem), the robustness with respect to the coefficient of the present method is an 
advantage.


\section{Boundary operators}\label{sec:b_op}
We define the Green's function for the Laplace operator in $\RR^3$ by
\begin{equation}
G(\x,\y) =
\frac{1}{4\uppi|\x-\y|}.
\end{equation}
In this paper, we focus on the problem in
$\RR^3$. Similar analysis can be used for problems in $\RR^2$, in which case this definition should be replaced by
$G(\x,\y) =-\log|\x-\y|/2\uppi$.

In the standard fashion (see eg \cite[chapter 6]{Stein07}), we define
the single layer potential operator, $\pop{V}:H^{-1/2}(\Gamma)\to H^1(\Omega)$,
and
the double layer potential operator, $\pop{K}:H^{1/2}(\Gamma)\to H^1(\Omega)$,
for $v\in H^{1/2}(\Gamma)$, $\mu\in H^{-1/2}(\Gamma)$, and $\x\in\Omega\setminus\Gamma$ by
\begin{align}
(\popV\mu)(\x) &:= \int_{\Gamma} G(\x,\y) \mu(\y)\dx[\y]\label{eq:single},\\
(\popK v)(\x) &:= \int_{\Gamma} \frac{\partial G(\x,\y)}{\partial\B\nu_\y} v(\y)\dx[\y]\label{eq:double}.
\end{align}
We define the space $H^1(\Delta,\Omega):=\{v\in H^1(\Omega):\Delta v\in\Ltwo(\Omega)\}$, and
then we define the Dirichlet and Neumann traces, $\trace\D:H^1(\Omega)\to H^{1/2}(\Gamma)$
and $\trace\N:H^1(\Delta,\Omega)\to H^{-1/2}(\Gamma)$, by
\begin{align}
\trace\D f(\x)&:=\lim_{\Omega\ni\y\to\x\in\Gamma}f(\y),\\
\trace\N f(\x)&:=\lim_{\Omega\ni\y\to\x\in\Gamma}\B\nu_\x\cdot\nabla f(\y).
\end{align}

We recall that if the Dirichlet and Neumann traces of a
harmonic function are known, then the potentials \cref{eq:single} and
\cref{eq:double} may be used to reconstruct the function in $\Omega$ using the
following relation.
\begin{equation}\label{eq:represent}
u = -\popK(\trace\D u) + \popV(\trace\N u).
\end{equation}

It is also known \cite[lemma 6.6]{Stein07} that for all $\mu \in H^{-1/2}(\Gamma)$, the function
\begin{equation}\label{eq:sing_rec}
u^{\popV}_\mu := \popV \mu
\end{equation}
satisfies $-\Delta u^{\popV}_\mu = 0$ and
\begin{equation}\label{eq:sing_rec_stab}
\Hnorm[\Omega]{1}{u^{\popV}_\mu} \leqslant c \Hnorm{-1/2}{\mu}.
\end{equation}
Similarly, for the double layer potential there holds \cite[lemma 6.10]{Stein07} that for all $v \in H^{1/2}(\Gamma)$, the function
\begin{equation}\label{eq:doub_rec}
u^{\popK}_v := \popK v
\end{equation}
satisfies $-\Delta u^{\popK}_v = 0$ and
\begin{equation}\label{eq:doub_rec_stab}
\Hnorm[\Omega]{1}{u^{\popK}_v} \leqslant c \Hnorm{1/2}{v}.
\end{equation}

We define
$\average{\trace\D f}$ and $\average{\trace\N f}$ 
to be the averages of the interior and exterior Dirichlet and Neumann traces of $f$.
We define the single layer, double layer, adjoint double layer, and hypersingular boundary integral operators,
$\bopV:H^{-1/2}(\Gamma)\to H^{1/2}(\Gamma)$,
$\bopK:H^{1/2}(\Gamma)\to H^{1/2}(\Gamma)$,
$\bopKadj:H^{-1/2}(\Gamma)\to H^{-1/2}(\Gamma)$, and
$\bopW:H^{1/2}(\Gamma)\to H^{-1/2}(\Gamma)$,
by
\begin{subequations}
\begin{align}
(\bopK v)(\x)&:=\average{\trace\D\popK v}(\x),&
(\bopV \mu)(\x)&:=\average{\trace\D\popV\mu}(\x),\\
(\bopW v)(\x)&:=-\average{\trace\N\popK v}(\x),&
(\bopKadj\mu)(\x)&:=\average{\trace\N\popV\mu}(\x),
\end{align}
\end{subequations}
where 
$\x\in\Gamma$,
$v\in H^{1/2}(\Gamma)$ and $\mu\in H^{-1/2}(\Gamma)$
\cite[chapter 6]{Stein07}.

The following coercivity results are known for the single layer and hypersingular operators in
$\RR^3$, where $\left\langle  \cdot, \cdot \right\rangle_\Gamma$ denotes the
$H^{1/2}(\Gamma)$--$H^{-1/2}(\Gamma)$ duality pairing.

\begin{lemma}[Coercivity of $\bopV$]\label{lemma:coercivity1}
There exists $\alpha_\bopV>0$ such that
\begin{align*}
\alpha_\bopV \Hnorm{-1/2}{\mu}^2 &\leqslant \left\langle \bopV \mu,\mu\right\rangle_{\Gamma},
&&\forall \mu \in H^{-1/2}(\Gamma).
\end{align*}
\end{lemma}
\begin{proof}\cite[theorem 6.22]{Stein07}.\end{proof}

\begin{lemma}[Coercivity of $\bopW$]\label{lemma:coercivity2}
There exists $\alpha_\bopW>0$ such that
\begin{align*}
\alpha_\bopW \Hnorm{1/2}{v}^2 &\leqslant \left\langle\bopW v,v\right\rangle_{\Gamma},
&&\forall v \in H_*^{1/2}(\Gamma),
\end{align*}
where $H_*^{1/2}(\Gamma)$ denotes the set of functions
$v\in H^{1/2}(\Gamma)$ such that $\overline{v} = 0$,
where $\displaystyle\overline{v}:=\frac{\langle v,1\rangle_\Gamma}{\langle1,1\rangle_\Gamma}$ is the average value of $v$.
From this it follows that
\begin{align*}
\alpha_\bopW \Hstarnorm{1/2}{v}^2 &\leqslant \left\langle\bopW v,v\right\rangle_{\Gamma},
&&\forall v \in H^{1/2}(\Gamma),
\end{align*}
where $\Hstarnorm{1/2}{\cdot}$ is defined, for $v\in H^{1/2}(\Gamma)$, by 
$\displaystyle\Hstarnorm{1/2}{v}:=\Hnorm{1/2}{v-\overline{v}}$.
\end{lemma}
\begin{proof}
\cite[theorem 6.24]{Stein07}.
\end{proof}

The following boundedness results are also known.

\begin{lemma}[Boundedness]\label{lemma:stability}
There exist $C_\bopV,C_\bopK,C_\bopKadj,C_\bopW>0$ such that
\begin{align*}
\text{i)}&&\Hnorm{1/2}{\bopV \mu} &\leqslant C_\bopV \Hnorm{-1/2}{\mu} &&\forall \mu \in H^{-1/2}(\Gamma),\\
\text{ii)}&&\Hnorm{1/2}{\bopK v} &\leqslant C_\bopK \Hnorm{1/2}{v} &&\forall v \in H^{1/2}(\Gamma),\\
\text{iii)}&&\Hnorm{-1/2}{\bopKadj \mu} &\leqslant C_\bopKadj \Hnorm{-1/2}{\mu} &&\forall \mu \in H^{-1/2}(\Gamma),\\
\text{iv)}&&\Hnorm{-1/2}{\bopW v} &\leqslant C_\bopW \Hnorm{1/2}{v} &&\forall v \in H^{1/2}(\Gamma).
\end{align*}
\end{lemma}
\begin{proof}
\cite[sections 6.2--6.5]{Stein07}.
\end{proof}

We define the Calder\'on projector by
\begin{equation}\label{eq:calder}
\bop{C}:=
\begin{pmatrix}
(1-\sigma)\bopI-\bopK & \bopV\\
\bopW & \sigma\bopI+\bopKadj
\end{pmatrix},
\end{equation}
where $\sigma$ is defined as in \cite[equation 6.11]{Stein07},
and recall that if $u$ is a solution of \cref{eq:Laplace} then it satisfies
\begin{equation}\label{eq:calder_id}
\bop{C}\begin{pmatrix}\trace\D u\\\trace\N u\end{pmatrix}=\begin{pmatrix}\trace\D u\\\trace\N u\end{pmatrix}.
\end{equation}

Taking the product of \cref{eq:calder_id} with two test functions, and using the fact that $\sigma=\frac12$ almost everywhere,
we arrive at the following equations.
\begin{align}
\left\langle \trace\D u,\mu \right\rangle_{\Gamma} &= \left\langle (\tfrac12\bopI - \bopK)
  \trace\D u,\mu\right\rangle_{\Gamma} + \left\langle \bopV\trace\N
  u,\mu\right\rangle_{\Gamma} &&\forall \mu \in H^{-1/2}(\Gamma),\label{eq:Calderon_1}\\
\left\langle \trace\N u,v \right\rangle_{\Gamma} &= \left\langle (\tfrac12\bopI + \bopKadj)
  \trace\N u,v \right\rangle_{\Gamma} + \left\langle \bopW \trace\D
  u,v\right\rangle_{\Gamma} &&\forall v \in H^{1/2}(\Gamma).\label{eq:Calderon_2}
\end{align}

For a more compact notation, we introduce $\lambda=\trace\N u$ and $u=\trace\D u$ and the Calder\'on form
\begin{multline}\label{eq:compact_form}
\form{C}[(u,\lambda),(v,\mu)]:=
\left\langle(\tfrac12\bopI-\bopK)
 u,\mu\right\rangle_{\Gamma} + \left\langle \bopV\lambda,\mu\right\rangle_{\Gamma}\\
+
\left\langle (\tfrac12\bopI + \bopKadj)
  \lambda,v \right\rangle_{\Gamma} + \left\langle \bopW u,v\right\rangle_{\Gamma}.
\end{multline}
We may then rewrite \cref{eq:Calderon_1} and \cref{eq:Calderon_2} as
\begin{equation}\label{eq:realtion}
\form{C}[(u,\lambda),(v,\mu)]=\left\langle u,\mu\right\rangle_\Gamma+\left\langle\lambda,v\right\rangle_\Gamma.
\end{equation}

We will also frequently use the multitrace form, defined by
\begin{equation}\label{eq:mult_trace}
\form{A}[(u,\lambda),(v,\mu)]:=
-\left\langle\bopK
 u,\mu\right\rangle_{\Gamma} + \left\langle \bopV\lambda,\mu\right\rangle_{\Gamma}
+
\left\langle \bopKadj
  \lambda,v \right\rangle_{\Gamma} + \left\langle \bopW u,v\right\rangle_{\Gamma}.
\end{equation}
Using this, we may rewrite \cref{eq:realtion} as 
\begin{equation}\label{eq:skewsym_relation}
\form{A}[(u,\lambda),(v,\mu)]=\tfrac12\left\langle u,\mu\right\rangle_\Gamma+\tfrac12\left\langle\lambda,v\right\rangle_\Gamma.
\end{equation}

To quantify the two traces we introduce the product space 
\[
\productspace{V} := \begin{cases}
H^{1/2}(\Gamma) \times H^{-1/2}(\Gamma)&\text{if }\Gamma\N\cup\Gamma\R=\emptyset,\\
H^{1/2}(\Gamma) \times \Ltwo(\Gamma)&\text{otherwise.}
\end{cases}
\]
The additional regularity on the flux variable is required later when imposing Neumann and Robin conditions.
We also introduce the associated norm
\[
\Vnorm{(v,\mu)}:= \Hnorm{1/2}{v}+\Hnorm{-1/2}{\mu}.
\]

Using the results in \cref{lemma:coercivity1,lemma:coercivity2,lemma:stability},
we obtain the continuity and coercivity of $\form{A}$.
\begin{lemma}[Continuity]\label{lemma:cont_cald}
There exists $C>0$ such that
\begin{align*}
\left|\form{A}[(w,\eta),(v,\mu)]\right| &\leqslant C \Vnorm{(w,\eta)}\Vnorm{(v,\mu)}&&\forall(w,\eta),(v,\mu)\in\productspace{V}.
\end{align*}
\end{lemma}
\begin{proof}
Use the stability results from \cref{lemma:stability}.
\end{proof}

\begin{lemma}[Coercivity]\label{lemma:coerciv_cald}
There exists $\alpha>0$ such that
\begin{align*}
\alpha\left(
\Hstarnorm{1/2}{v}^2+\Hnorm{-1/2}{\mu}^2
\right)
&\leqslant
\form{A}[(v,\mu),(v,\mu)]
&&\forall(v,\mu)\in\productspace{V}.
\end{align*}
\end{lemma}
\begin{proof}
Use the coercivity of $\bopV$ and $\bopW$ from \cref{lemma:coercivity1,lemma:coercivity2} and let $\alpha=\min(\alpha_\bopW,\alpha_\bopV)$.
\end{proof}


\section{Weak Imposition of boundary conditions}\label{sec:weak}
In this section, we will derive boundary integral formulations of the
problem \cref{eq:Laplace}, that we will then use for our boundary element formulations.
We assume that the boundary condition may be written as 
\begin{equation}
R_\Gamma(u,\lambda) = 0.
\end{equation}

The idea that we will exploit in the following is simply to add a suitable weighted weak form of
this constraint to the Calder\'on form
\cref{eq:realtion}. Formally, this leads to an expression of the
form
\begin{equation}\label{eq:cald_bc}
\form{C}[(u,\lambda),(v,\mu)]=\left\langle u,\mu
\right\rangle_\Gamma+\left\langle \lambda, v\right\rangle_{\Gamma} +
\left\langle R_\Gamma(u,\lambda), \betaparam_1 v + \betaparam_2 \mu \right\rangle_{\Gamma},
\end{equation}
or equivalently
\begin{equation}\label{eq:multi_bc}
\form{A}[(u,\lambda),(v,\mu)]=\tfrac12\left\langle u,\mu
\right\rangle_\Gamma+\tfrac12\left\langle \lambda, v\right\rangle_{\Gamma} +
\left\langle R_\Gamma(u,\lambda), \betaparam_1 v + \betaparam_2 \mu \right\rangle_{\Gamma},
\end{equation}
where $\betaparam_1$ and $\betaparam_2$ are problem dependent scaling operators that will be chosen as a
function of the physical parameters in order to obtain robustness of
the method.


\subsection{Dirichlet boundary condition}\label{sec:dirichlet}
In this section, we assume that $\Gamma\D\equiv\Gamma$ and consider the resulting Dirichlet problem.
We choose $\betaparam_1 = \betaparam\D^{1/2}$, $\betaparam_2 = \betaparam\D^{-1/2}$, where $\betaparam\D$ will
be identified with a mesh-dependent penalty parameter, and
\begin{equation}\label{eq:Dir_R}
R_{\Gamma\D}(u,\lambda) := \betaparam\D^{1/2} (g\D-u)
\end{equation}
where $g\D \in H^{1/2}(\Gamma)$ is the Dirichlet data.

Inserting this into \cref{eq:multi_bc}, we obtain the formulation
\begin{multline}\label{eq:multi_D}
\form{A}[(u,\lambda),(v,\mu)] -\tfrac12 \left\langle  \lambda, v
\right\rangle_{\Gamma\D} + \tfrac12
\left\langle u, \mu \right\rangle_{\Gamma\D} +\left\langle \betaparam\D  u, v 
\right\rangle_{\Gamma\D} = \left\langle g\D, \betaparam\D v + \mu
\right\rangle_{\Gamma\D}.
\end{multline}
One can compare the method with the classical (non-symmetric) Nitsche's method
by formally identifying $\lambda$ with $\partial_\B\nu u$ and $\mu$ with
$\partial_\B\nu v$ (up to the multiplicative factor $\tfrac12$).

For a more compact notation, we introduce the boundary operator
associated with the non-homogeneous Dirichlet condition
\begin{equation}\label{eq:operator_Dirichlet}
\form{B}\D[(u,\lambda),(v,\mu)]:=-\tfrac12\left\langle  \lambda, v \right\rangle_{\Gamma\D} +
\tfrac12\left\langle u, \mu \right\rangle_{\Gamma\D} +\left\langle \betaparam\D u, v 
\right\rangle_{\Gamma\D},
\end{equation}
and the operator associated with the right hand side
\begin{equation}\label{eq:LDir} 
\form{L}\D(v,\mu) := \left\langle  g\D,\betaparam\D v + \mu\right\rangle_{\Gamma\D}.
\end{equation}

Using these and \cref{eq:multi_D}, we arrive at the following problem: Find
$(u,\lambda) \in \productspace{V}$ such that
\begin{align}\label{eq:abstract_form_Dir}
\form{A}[(u,\lambda),(v,\mu)]+\form{B}\D[(u,\lambda),(v,\mu)] &=
\form{L}\D(v,\mu)&&\forall(v,\mu) \in \productspace{V}.
\end{align}

If we set $\betaparam\D=0$ in \cref{eq:operator_Dirichlet,eq:LDir}, we obtain a penalty-free formulation for the Dirichlet problem.


\subsection{Neumann boundary condition}\label{sec:neumann}
In this section, we assume that $\Gamma\N\equiv\Gamma$ and consider the resulting Neumann problem.
We choose $\betaparam_1 = \betaparam\N^{-1/2}$, $\betaparam_2 = \betaparam\N^{1/2}$, and define
\begin{equation}\label{eq:Neu_R}
R_{\Gamma\N}(u,\lambda) := \betaparam\N^{1/2} (g\N - \lambda),
\end{equation}
where $g\N \in \Ltwo(\Gamma\N)$, with $\int_\Gamma g\N=0$, is the Neumann data.

Proceeding as in the Dirichlet case, we obtain the formulation
\begin{multline}\label{eq:multi_N}
\form{A}[(u,\lambda),(v,\mu)] -\tfrac12 \left\langle  u, \mu
\right\rangle_{\Gamma\N} + \tfrac12
\left\langle \lambda, v \right\rangle_{\Gamma\N} +\left\langle \betaparam\N  \lambda, \mu 
\right\rangle_{\Gamma\N} = \left\langle g\N, \betaparam\N \mu + v
\right\rangle_{\Gamma\N}.
\end{multline}
Defining
\begin{align}\label{eq:operator_Neumann}
\form{B}\N[(u,\lambda),(v,\mu)]&:=
-\tfrac12\left\langle u, \mu \right\rangle_{\Gamma\N}
+\tfrac12\left\langle  \lambda, v \right\rangle_{\Gamma\N}
+\left\langle \betaparam\N\lambda,\mu\right\rangle_{\Gamma\N},\\
\label{eq:LNeu} 
\form{L}\N(v,\mu) &:= \left\langle  g\N,\betaparam\N \mu + v\right\rangle_{\Gamma\N},
\end{align}
we may write this as the variational problem: 
Find $(u,\lambda) \in \zerostar{\productspace{V}}$ such that
\begin{align}\label{eq:abstract_form_Neu}
\form{A}[(u,\lambda),(v,\mu)]+\form{B}\N[(u,\lambda),(v,\mu)] &=
\form{L}\N(v,\mu)&&\forall(v,\mu) \in \zerostar{\productspace{V}}.
\end{align}
Here, we use the space $\zerostar{\productspace{V}}:=H_*^{1/2}(\Gamma\N)\times\Ltwo(\Gamma\N)$, as the solution to the Neumann problem can only be
determined up to a constant, so we include the extra condition that $\overline{u}=0$.

If we set $\betaparam\N=0$ in \cref{eq:operator_Neumann,eq:LNeu}, we obtain a penalty-free formulation for the Neumann problem.
In this case, we may take
$\zerostar{\productspace{V}}=H_*^{1/2}(\Gamma\N)\times H^{-1/2}(\Gamma\N)$ and
$g\N\in H^{-1/2}(\Gamma\N)$.

When $\betaparam\N>0$, observe that for the terms imposing the Neumann condition to be well
defined, we need $\lambda \in \Ltwo(\Gamma\N)$. This can be avoided by
replacing $\betaparam\N$
with a regularising operator $\bop{R}:H^{-1/2}(\Gamma\N) \to H^{1/2}(\Gamma\N)$. For example, 
we could take $\bop{R}=\betaparam_\bopV \bopV$, where
$\betaparam_\bopV \in \RR$ and $\bopV$ is the single layer boundary operator on $\Gamma\N$.
This formulation with the operator $\bop{R}$ is given in \cite[(3.10) and (3.11)]{Steinbach2011}, where it was derived
using a domain decomposition approach where a Robin condition was used to weakly impose a Neumann condition.

The resulting formulations using $\betaparam\N$ are in general easier to analyse, since
they give control of $\lambda$ on the Neumann boundary in the natural
norm $\Hnorm[\Gamma\N]{-1/2}{\lambda}$.


\subsection{Mixed Dirichlet--Neumann boundary condition}\label{sec:mixed_neumann_dirichlet}
We now consider the case of mixed Dirichlet--Neumann boundary conditions, when $\Gamma = \Gamma\D \cup \Gamma\N$. We note that in this 
case, and in the Robin case, we take $\productspace{V} = H^{1/2}(\Gamma) \times \Ltwo(\Gamma)$.

Let $R_{\Gamma\D}$ and $R_{\Gamma\N}$ be defined by
\cref{eq:Dir_R} and \cref{eq:Neu_R}.
Using the abstract form \cref{eq:multi_bc}, we obtain
\begin{multline}\label{eq:mixed_bc}
\form{A}[(u,\lambda),(v,\mu)]=\tfrac12\left\langle u,\mu
\right\rangle_\Gamma+\tfrac12\left\langle \lambda, v\right\rangle_{\Gamma} \\
+
\left\langle R_{\Gamma\D}(u,\lambda), \betaparam\D^{1/2} v + \betaparam\D^{-1/2} \mu
\right\rangle_{\Gamma\D} + \left\langle R_{\Gamma\N}(u,\lambda), \betaparam\N^{-1/2} v + \betaparam\N^{1/2} \mu
\right\rangle_{\Gamma\N}.
\end{multline}
Developing \cref{eq:mixed_bc}, and defining
\begin{multline}\label{eq:operator_Mixed}
\form{B}\ND[(u,\lambda),(v,\mu)] :=
\tfrac12\left\langle u, \mu \right\rangle_{\Gamma\D}
    - \tfrac12\left\langle\lambda, v \right\rangle_{\Gamma\D}
    + \left\langle\betaparam\D u, v \right\rangle_{\Gamma\D} \\
+
\tfrac12\left\langle\lambda, v \right\rangle_{\Gamma\N}
    - \tfrac12\left\langle u, \mu \right\rangle_{\Gamma\N} 
    + \left\langle\betaparam\N \lambda, \mu \right\rangle_{\Gamma\N},
\end{multline}
\begin{equation}\label{eq:LMix}
\form{L}\ND(v,\mu) := \left\langle g\D, \betaparam\D v + \mu
\right\rangle_{\Gamma\D}+ \left\langle  g\N,\betaparam\N\mu + v\right\rangle_{\Gamma\N},
\end{equation}
we arrive the variational formulation: Find $(u,\lambda) \in
\productspace{V}$
such that
\begin{align}\label{eq:MND_compact}
\form{A}[(u,\lambda),(v,\mu)] +
\form{B}\ND[(u,\lambda),(v,\mu)] &= \form{L}\ND(v,\mu)
&&\forall(v,\mu) \in
\productspace{V}.
\end{align}

If we set $\betaparam\D=0$ and $\betaparam\N=0$ in \cref{eq:operator_Mixed,eq:LMix}, we obtain a penalty-free formulation for the mixed Dirichlet--Neumann problem.
By taking $\Gamma\N=\emptyset$ or $\Gamma\D=\emptyset$, formulations for both Dirichlet and Neumann problems can be obtained from \cref{eq:MND_compact}.


\subsection{Robin conditions}
For simplicity, we consider the case where $\Gamma=\Gamma\R$.
Considering the Robin condition \cref{eq:LaplaceRobin}, we may write, for some $\param>0$,
\begin{equation}
R_{\Gamma\R}(u,\lambda) := \betaparam\R^{1/2}\left(\param^{1/2}(g\N-\lambda) + \param^{-1/2}(g\D-u)\right).
\end{equation}
This function is a linear combination of the Dirichlet and the Neumann conditions.
\begin{equation}\label{eq:Robin_bc}
R_{\Gamma\R}(u,\lambda) = \alpha\D R_{\Gamma\D}(u,\lambda) + \alpha\N R_{\Gamma\N}(u,\lambda),
\end{equation}
where $\alpha\N=\betaparam\R^{1/2}\betaparam\N^{-1/2}\param^{1/2}$ and $\alpha\D=\betaparam\R^{1/2}\betaparam\D^{-1/2}\param^{-1/2}$.

We take $\betaparam_1=\betaparam\R^{1/2}$ and $\betaparam_2=\betaparam\R^{-1/2}$, and look for a term of the form
\begin{equation}
\left\langle \phi R_{\Gamma\R}(u,\lambda) ,
  \betaparam\R^{1/2} v + \betaparam\R^{-1/2} \mu \right\rangle_{\Gamma\R},
\end{equation}
where the $\phi$ and $\betaparam\R$ must have the following properties to ensure that 
the formulation degenerates into the formulation for the Dirichlet and Neumann problems as $\param\to0$ and $\param\to\infty$.
\begin{align*}
&\betaparam\R\to\betaparam\D,
&&\alpha\D \phi \to 1,
&&\text{and}
&&\alpha\N \phi \to 0
&&\ 
&&\text{as }\param\to 0,\\
&\betaparam\R\to\betaparam\N^{-1},
&&\alpha\N \phi \to 1,
&&\text{and}
&&\alpha\D \phi \to 0
&&\ 
&&\text{as }\param\to \infty.
\end{align*}

It is straightforward to verify that these conditions are satisfied
for the choices
\begin{align}
\phi&:= \frac{\param^{1/2}}{\param\betaparam\R + 1},\\
\betaparam\R&:=\frac{\param\betaparam\N^{-1}+\betaparam\D}{\param+1}.
\end{align}
Later, we will use 
$\betaparam\D=\betaparam h^{-1}$
and
$\betaparam\N=\betaparam h$,
where $\betaparam$ is a constant, as in the mixed Dirichlet--Neumann case.

Collecting the above considerations, we arrive at the formulation
\begin{multline}\label{eq:cald_Rob}
\form{A}[(u,\lambda),(v,\mu)] =
\tfrac12\left\langle u,\mu \right\rangle_{\Gamma}+\tfrac12\left\langle \lambda, v\right\rangle_{\Gamma} \\
+ \left\langle
    \param(g\N-\lambda) + (g\D-u),
    \frac{\betaparam\R}{\param \betaparam\R+1} v +  \frac{1}{\param \betaparam\R+1} \mu
\right\rangle_{\Gamma\R}.
\end{multline}
Taking $\param\to0$, we recover the Dirichlet formulation \cref{eq:multi_D};
and taking $\param \to \infty$ results in the Neumann formulation  \cref{eq:multi_N}.

By introducing 
\begin{multline*}
\form{B}\R[(u,\lambda),(v,\mu)]:=
\frac12\left\langle \frac{\param
    \betaparam\R -1}{\param
    \betaparam\R +1} \lambda, v\right\rangle_{\Gamma\R}
- \frac12\left\langle \frac{\param
    \betaparam\R -1}{\param
    \betaparam\R +1} u, \mu\right\rangle_{\Gamma\R} \\
+\left\langle \frac{\param}{\param
    \betaparam\R +1} \lambda, \mu \right\rangle_{\Gamma\R}
+\left\langle \frac{\betaparam\R}{\param
    \betaparam\R +1} u,v\right\rangle_{\Gamma\R} 
\end{multline*}
and
\[
\form{L}\R(v,\mu)
:=
\left\langle g\D+ \param g\N, \frac{\betaparam\R}{\param \betaparam\R +1} v  +
  \frac{1}{\param \betaparam\R +1} \mu\right\rangle_{\Gamma\R},
\]
we may write this as the variational problem: Find $(u,\lambda) \in\productspace{V}$ such that
\begin{align}\label{eq:cald_Rob_compact}
\form{A}[(u,\lambda),(v,\mu)] +
\form{B}\R[(u,\lambda),(v,\mu)] &= \form{L}\R(v,\mu) &&\forall
(v,\mu) \in\productspace V.
\end{align}


\section{Boundary element method for the single domain problem}\label{sec:bem}
All the methods introduced above are written as the sum of the multitrace operator $\form{A}$ and a
boundary condition operator $\form{B}$. We write this generally as: Find
$(u,\lambda) \in \productspace{V}$ such that
\begin{align}\label{eq:abstract_form}
\form{A}[(u,\lambda),(v,\mu)]+\form{B}[(u,\lambda),(v,\mu)] &=
\form{L}(v,\mu) &&\forall (v,\mu) \in \productspace{V}.
\end{align}
In this section, we analyse this general problem, then show that the analysis is applicable to the
boundary conditions discussed in \cref{sec:weak}.

For the sake of example and to fix the ideas, we introduce a family of conforming, shape regular triangulations of
$\Gamma$, $\{\Th\}_{h>0}$, indexed by the largest element
diameter of the mesh, $h$.
We assume that the triangulations are fitted to the different
boundary sets $\Gamma\D$, $\Gamma\R$ and $\Gamma\N$.
We then consider the following finite element spaces.
\begin{align*}
\Vee^k_h &:= \{v_h \in C^0(\Gamma): v_h\vert_{T} \in \mathbb{P}_k(T) \text{, for every }T\in\Th\},\\
\Lambda^l_h &:= \{v_h \in \Ltwo(\Gamma): v_h\vert_{T} \in \mathbb{P}_l(T) \text{, for every }T\in\Th\},\\
\tilde \Lambda^l_h &:= \{v_h \in \Lambda^l_h: v_h\vert_{\Gamma_i} \in
C^0(\Gamma_i)\text{, for }i=1,\hdots,M\},
\end{align*}
where $\mathbb{P}_k(T)$ denotes the space of polynomials of order less
than or equal to $k$,
and $\{\Gamma_i\}_{i=1}^M$ are the polygonal faces of $\Gamma$.

We observe that $\Vee^k_h \subset H^{1/2}(\Gamma)$, $\Lambda^l_h
\subset \Ltwo(\Gamma)$ and $\tilde \Lambda^l_h
\subset \Ltwo(\Gamma)$. We now introduce the discrete product
space $\productspace{V}_h :=  \Vee^k_h \times \Lambda^l_h$. The space $\tilde
\Lambda^l_h$ may be used in the place of $\Lambda^l_h$ without any modifications of the arguments below.

The boundary element formulation of the generic problem \cref{eq:abstract_form} then takes the form: Find
$(u_h,\lambda_h) \in  \productspace{V}_h$ such that
\begin{align}\label{BEMform}
\form{A}[(u_h,\lambda_h),(v_h,\mu_h)] +
\form{B}[(u_h,\lambda_h),(v_h,\mu_h)] &= \form{L}(v_h,\mu_h)
&&\forall (v_h, \mu_h) \in \productspace{V}_h.
\end{align}

If we assume that
$(u,\lambda) \in \productspace{V}$ and $(u_h,\lambda_h) \in \productspace{V}_h$
satisfy
\cref{eq:abstract_form} and \cref{BEMform},
it immediately follows that the following Galerkin orthogonality relation holds.
\begin{multline}\label{galortho}
\form{A}[(u-u_h,\lambda-\lambda_h),(v_h,\mu_h)] +
\form{B}[(u-u_h,\lambda-\lambda_h),(v_h,\mu_h)] = 0
\\
\forall (v_h, \mu_h) \in \productspace{V}_h.
\end{multline}
We also get the following representation formula for the approximation in the
bulk using \cref{eq:represent}.
\begin{equation}\label{eq:rec_disc}
\tilde u_h = -\popK u_h + \popV\lambda_h.
\end{equation}
We will now proceed to derive some estimates for the solution of
\cref{BEMform} and the reconstruction \cref{eq:rec_disc}.

Let $\productspace{W}$ be a product Hilbert space for the primal and flux
variables, such that $\productspace{W} \subset \productspace{V}$.
Let $\Bnorm{\cdot}$ be a norm defined on $\productspace{W}$, such that for all $(v,\mu)\in\productspace{W}$,
$\Bnorm{(v,\mu)} \geqslant \Vnorm{(v,\mu)}$.

To reduce the number of constants that appear, especially when proving that \cref{a:approx} holds, we introduce the following notation.
\begin{itemize}
\item If $\exists C>0$, independent of $h$, such that $a\leqslant Cb$, then we write $a\lesssim b$.
\item If $a\lesssim b$ and $b\lesssim a$, then we write $a\eqsim b$.
\end{itemize}
For the abstract analysis, we will make use of the following standard assumptions.

\begin{assumption}[Weak coercivity]\label{a:coercive}
There exists $\alpha>0$ such that $\forall (v,\mu) \in \productspace{W}$
\[\alpha \Bnorm{(v,\mu)} \leqslant \sup_{(w,\eta) \in \productspace{W}\setminus\{0\}}\frac{\form{A}[(v,\mu),(w,\eta)] +
\form{B}[(v,\mu),(w,\eta)]}{\Bnorm{(w,\eta)}},\]
and $\forall (w,\eta) \in \productspace{W}\setminus\{0\}$
\[\sup_{(v,\mu) \in \productspace{W}}\left|\form{A}[(v,\mu),(w,\eta)] +
\form{B}[(v,\mu),(w,\eta)]\right|>0.\]
\end{assumption}

\begin{assumption}[Discrete coercivity]\label{a:discoercive}
There exists $\alpha>0$ such that $\forall (v_h,\mu_h) \in \productspace{V}_h$
\[\alpha \Bnorm{(v_h,\mu_h)} \leqslant \sup_{(w_h,\eta_h) \in \productspace{V}_h\setminus\{0\}}\frac{\form{A}[(v_h,\mu_h),(w_h,\eta_h)] +
\form{B}[(v_h,\mu_h),(w_h,\eta_h)]}{\Bnorm{(w_h,\eta_h)}},\]
and $\forall (w_h,\eta_h) \in \productspace{V}_h\setminus\{0\}$
\[\sup_{(v_h,\mu_h) \in \productspace{V}_h}\left|\form{A}[(v_h,\mu_h),(w_h,\eta_h)] +
\form{B}[(v_h,\mu_h),(w_h,\eta_h)]\right|>0.\]
\end{assumption}

\begin{assumption}[Continuity]\label{a:continuity}
There exists an auxiliary norm $\starnorm{(v,\mu)}$ defined on $\productspace{W}$, and
there exists
$M>0$ such that $\forall(w,\eta),(v,\mu) \in \productspace{W}$
\[\left|\form{A}[(w,\eta),(v,\mu)] +
\form{B}[(w,\eta),(v,\mu)]\right| \leqslant M\starnorm{(w,\eta)}\Bnorm{(v,\mu)}
\]
\end{assumption}

\begin{assumption}[Approximation]\label{a:approx}
$\forall (v,\mu) \in H^s(\Gamma)\times H^r(\Gamma)$,
\begin{equation*}
\inf_{(w_h,\eta_h) \in \productspace{V}_h}
    \starnorm{(v-w_h,\mu - \eta_h)}
\lesssim h^{\zeta-1/2} \Hseminorm{\zeta}{v} +
h^{\xi+1/2} \Hseminorm{\xi}{\mu},
\end{equation*}
where $\zeta = \min(k+1,s)$, $\xi = \min(l+1,r)$,
$s\geqslant\frac12$ and $r\geqslant-\frac12$.
\end{assumption}

Typically, we use approximation spaces with $k=l+1$, where the polynomial spaces used for $\lambda$ are one order lower than those for $u$,
or spaces with $k=l$, where equal order spaces are used for both variables.

We note that if the form $\form{A}+\form{B}$ is coercive,
that is there exists $\alpha>0$ such that $\forall (v,\mu) \in \productspace{W}$
\[\alpha \Bnorm{(v,\mu)}^2 \leqslant \form{A}[(v,\mu),(v,\mu)] + \form{B}[(v,\mu),(v,\mu)],\]
then \cref{a:coercive,a:discoercive} hold.

We now proceed to prove some results about the abstract problem.
\begin{proposition}\label{prop:exist_uniqueness}
Assume that \cref{a:coercive} holds, then the linear system defined by \cref{BEMform} is invertible.
If, in addition, we assume that
\begin{itemize}
\item\cref{a:continuity} holds,
\item there exists $L>0$ such that $\form{L}(w,\eta) \leqslant L \Bnorm{(w,\eta)}\quad\forall(w,\eta) \in \productspace{W}$,
\item and $\starnorm{\cdot}$ is equivalent to $\Bnorm{\cdot}$,
\end{itemize}
then the formulation \cref{eq:abstract_form} admits a unique solution
in $\productspace{W}$.
\end{proposition}
\begin{proof}
Note that \cref{a:coercive} implies the inf-sup condition,
\begin{align}
\inf_{(v,\mu)\in\productspace{W}\setminus\{0\}}\sup_{(w,\eta)\in\productspace{W}\setminus\{0\}}
\frac{\form{A}[(v,\mu),(w,\eta)] + \form{B}[(v,\mu),(w,\eta)]}{
\Bnorm{(v,\mu)}\Bnorm{(w,\eta)}
}>0
.
\end{align}
Therefore we may apply the Babu\v{s}ka--Lax--Milgram theorem \cite[theorem 5.2.1]{AzizBab}.
\end{proof}
\begin{proposition}\label{prop:best_approximation}
Assume that $(u,\lambda) \in \productspace{V}$ is the solution to a boundary
value problem of the form \cref{eq:Laplace} satisfying the
abstract form \cref{eq:abstract_form}. Let $(u_h,\lambda_h) \in
\productspace{V}_h$ be the solution of \cref{BEMform}.
If \cref{a:discoercive,a:continuity} are
satisfied
then 
\begin{equation}\label{eq:best}
\Bnorm{(u-u_h,\lambda-\lambda_h)} \leqslant\frac{M}\alpha
\inf_{(v_h,\mu_h)\in \productspace{V}_h}
\starnorm{(u-v_h,\lambda-\mu_h)}.
\end{equation}
\end{proposition}
\begin{proof}
See \cite[theorem 2]{Xu}.
\end{proof}

\begin{corollary}\label{col:error_est}
Let $(u,\lambda) \in H^s(\Gamma) \times H^r(\Gamma)$, for some $s\geqslant\frac12$ and $r\geqslant-\frac12$, satisfy the
abstract form \cref{eq:abstract_form}.
Under the assumptions of \cref{prop:best_approximation} and \cref{a:approx},
\[
\Bnorm{(u-u_h,\lambda-\lambda_h)} \lesssim
h^{\zeta-1/2} \Hseminorm{\zeta}{u}
+ h^{\xi+1/2}\Hseminorm{\xi}{\lambda},
\]
where $\zeta = \min(k+1,s)$ and $\xi = \min(l+1,r)$.
\end{corollary}
\begin{proof}
Apply \cref{a:approx} to the right hand side of \cref{eq:best}.
\end{proof}

\begin{proposition}\label{prop:bulk_error}
Assume that $(u,\lambda) \in \productspace{V}$ is the solution to a boundary
value problem of the form \cref{eq:Laplace} satisfying the
abstract form \cref{eq:abstract_form} and that the assumptions of
\cref{prop:best_approximation} are satisfied. Let $(u_h,\lambda_h) \in
\productspace{V}_h$. Let $\tilde u:\Omega \to \RR$ be the
reconstruction obtained using \cref{eq:represent}, with $\trace\N
u=\lambda$ and $\trace\D u=u$; and $\tilde u_h:\Omega \to \RR$
be the reconstruction obtained using \cref{eq:rec_disc}. Then there
holds
\[
\Hnorm[\Omega]{1}{\tilde u-\tilde u_h} \lesssim \frac{M}{\alpha} \inf_{v_h,\mu_h
  \in \productspace{V}_h} \starnorm{(u-v_h,\lambda-\mu_h)}.
\]
\end{proposition}
\begin{proof}
Using \cref{eq:sing_rec} and \cref{eq:doub_rec}, we may write
\[
\tilde u - \tilde u_h = (u^\popV_\lambda - u^\popV_{\lambda_h})+(u^\popK_u - u^\popK_{u_h}).
\]
Using the triangle inequality, we have
\begin{equation}\label{eq:triangle}
\Hnorm[\Omega]{1}{\tilde u-\tilde u_h} \leqslant \Hnorm[\Omega]{1}{u^\popV_\lambda - u^\popV_{\lambda_h}}+\Hnorm[\Omega]{1}{u^\popK_u - u^\popK_{u_h}}.
\end{equation}
By \cref{eq:sing_rec_stab} and \cref{eq:doub_rec_stab}, there exist $c_1,c_2>0$ such that
\begin{align}
\Hnorm[\Omega]{1}{u^\popV_\lambda - u^\popV_{\lambda_h}} &\leqslant c_1\Hnorm{-1/2}{\lambda - \lambda_h},\label{eq:staberr1}\\
\Hnorm[\Omega]{1}{u^\popK_u- u^\popK_{u_h}} &\leqslant c_2\Hnorm{1/2}{u - u_h}.\label{eq:staberr2}
\end{align}
Collecting \cref{eq:triangle,eq:staberr1,eq:staberr2}, we see that there exists $C>0$ such that
\begin{equation}\label{eq:final_line_of_proof}
\Hnorm[\Omega]{1}{\tilde u-\tilde u_h} \leqslant C\Vnorm{\lambda - \lambda_h,u-u_h}\leqslant C\Bnorm{\lambda - \lambda_h,u-u_h}.
\end{equation}
The statement now follows from \cref{prop:best_approximation}.
\end{proof}

\begin{corollary}\label{cor:bulk_est}
Under the same assumptions of \cref{prop:bulk_error} and \cref{a:approx},
\[
\Hnorm[\Omega]{1}{\tilde u-\tilde u_h} \lesssim
h^{\zeta-1/2} \Hseminorm{\zeta}{u}
+ h^{\xi+1/2}\Hseminorm{\xi}{\lambda},
\]
where $\zeta = \min(k+1,s)$ and $\xi = \min(l+1,r)$.
\end{corollary}
\begin{proof}
Apply \cref{a:approx} to \cref{eq:final_line_of_proof} in the proof of \cref{prop:bulk_error}.
\end{proof}

\subsection{Application of the theory to the Dirichlet problem}\label{sec:bem_dirichlet}
For the finite element spaces defined above, the Dirichlet problem
takes the form: Find
$(u_h,\lambda_h) \in\productspace{V}_h$ such that
\begin{align}\label{BEMform_Dir}
\form{A}[(u_h,\lambda_h),(v_h,\mu_h)] +
\form{B}\D[(u_h,\lambda_h),(v_h,\mu_h)] &= \form{L}\D(v_h,\mu_h)
&&\forall(v_h,\mu_h) \in \productspace{V}_h.
\end{align}

We introduce the following $\form{B}\D$-norm.
\[
\Bnorm[\D]{(v,\mu)} := \Vnorm{(v,\mu)} + \betaparam\D^{1/2} \Ltwonorm[\Gamma\D]{v},
\]
and we let $\starnorm{\cdot}=\Bnorm[\D]{\cdot}$.
We now proceed to verify that \cref{a:coercive,a:discoercive,a:continuity,a:approx} hold.

\begin{proposition}[Coercivity]\label{diri_coercive}
\Cref{a:coercive,a:discoercive} are satisfied for the Dirichlet problem if $\exists\betamin>0$, independent of $h$,
such that $\betaparam\D>\betamin$.
\end{proposition}

\begin{proof}
Using the fact that $\Hstarnorm[\Gamma\D]{1/2}{v}^2+\Ltwonorm[\Gamma\D]{\overline{v}}^2\gtrsim\Hnorm[\Gamma\D]{1/2}{v}^2$, we deduce from \cref{lemma:coerciv_cald}
that for every positive $\alpha'\leqslant\alpha$,
\begin{align*}
\alpha'
\Vnorm{(v,\mu)}^2
-\alpha'\Ltwonorm[\Gamma\D]{\overline{v}}^2
&\leqslant
\form{A}[(v,\mu),(v,\mu)]
&&\forall(v,\mu)\in\productspace{W}.
\end{align*}
Using the definition of $\form{B}\D$, we see that
\begin{align*}
\form{B}\D[(v,\mu),(v,\mu)]&=\betaparam\D\langle v,v\rangle_{\Gamma\D}=\betaparam\D\Ltwonorm[\Gamma\D]{v}^2
\end{align*}
Taking $\alpha'=\min(\alpha,\betamin/2)$, we see that
\begin{align*}
\form{A}[(v,\mu),(v,\mu)]+\form{B}\D[(v,\mu),(v,\mu)]
&\geqslant\alpha'\Vnorm{(v,\mu)}^2+\left(1-\frac{\alpha'}{\betamin}\right)\betaparam\D\Ltwonorm[\Gamma\D]{v}^2\\
&\geqslant\alpha''\Bnorm[\D]{(v,\mu)}^2,
\end{align*}
for some $\alpha''>0$.
Therefore, in this case the form $\form{A}+\form{B}\D$ is coercive, and so \cref{a:coercive,a:discoercive} hold.
\end{proof}

\begin{proposition}[Weak coercivity]
\Cref{a:coercive,a:discoercive} are satisfied for the Dirichlet problem with $\betaparam\D=0$.
\end{proposition}

\begin{proof}
Taking $w=v$ and $\eta=\mu+c\overline{v}$, for some $c\in\RR$ to be fixed, we obtain
\begin{align}
L&:=\form{A}[(v,\mu),(w,\eta)]+
\form{B}\D[(v,\mu),(w,\eta)]\nonumber\\
&=
\left\langle \bopV \mu,\mu\right\rangle_{\Gamma}
+ c\left\langle \bopV \mu,\overline{v}\right\rangle_{\Gamma}
- c\left\langle \bopK v,\overline{v} \right\rangle_{\Gamma}
+ \left\langle \bopW v,v\right\rangle_{\Gamma}
+ \frac{c}2\left\langle v,\overline{v}\right\rangle_\Gamma\label{eq:dirichlet_zero_first_bit}
\end{align}
By \cref{lemma:coercivity1,lemma:coercivity2}, we know that
\begin{align}\label{diricostep1}
\left\langle \bopV \mu,\mu\right\rangle_{\Gamma}
+ \left\langle \bopW v,v\right\rangle_{\Gamma}
\geqslant
\alpha_\bopV\Hnorm{-1/2}{\mu}^2
+\alpha_\bopW\Hstarnorm{1/2}{v}^2
\end{align}
By \cref{lemma:stability}, we see that
\begin{align*}
c\left|\left\langle \bopV\mu, \overline{v}\right\rangle_{\Gamma}\right|
&\leqslant
c\Hnorm{1/2}{\bopV \mu}\Hnorm{-1/2}{\overline{v}}\\
&\leqslant
cC_\bopV\Hnorm{-1/2}{\mu}\Hnorm{-1/2}{\overline{v}}\\
&=
cC_\bopV\Hnorm{-1/2}{\mu}\Ltwonorm{\overline{v}}
\end{align*}
Using the fact that for $a,b\geqslant0$, $ab\leqslant(a^2+b^2)/2$, we obtain
\begin{align}
c\left|\left\langle \bopV\mu, \overline{v}\right\rangle_{\Gamma}\right|
&\leqslant
\frac{c^2C_\bopV^2}{2\alpha_\bopV}\Ltwonorm{\overline{v}}^2+\frac{\alpha_\bopV}{2}\Hnorm{-1/2}{\mu}^2.
\label{diricostep2}
\end{align}
We note that $u=\overline{v}$ is a solution to \cref{eq:Laplace}, $\trace\D\overline{v}=\overline{v}$ and $\trace\N \overline{v}=0$. Using this and applying \cref{eq:Calderon_1}, we see that
$\forall\mu\in H^{-1/2}(\Gamma)$, $\left\langle\bopK\overline{v},\mu\right\rangle_{\Gamma}=-\tfrac12\left\langle\overline{v},\mu\right\rangle_{\Gamma}$.
Therefore, using $\mu=\overline{v}$,
\begin{align*}
c\left\langle \bopK v, \overline{v} \right\rangle_{\Gamma}
&=c\left\langle \bopK (v-\overline{v}), \overline{v} \right\rangle_{\Gamma}
+c\left\langle \bopK \overline{v}, \overline{v} \right\rangle_{\Gamma}
\\
&=c\left\langle \bopK (v-\overline{v}), \overline{v} \right\rangle_{\Gamma}
-\frac{c}2\left\langle \overline{v}, \overline{v} \right\rangle_{\Gamma}.
\end{align*}
Using the fact that $\Hnorm{1/2}{v-\overline{v}}=\Hstarnorm{1/2}{v}$, and
proceeding in the same way as we did for the single layer term above, we obtain
\begin{align}
c\left\langle \bopK v,\overline{v} \right\rangle_{\Gamma}
&\leqslant
\frac{\alpha_\bopW}{2}\Hstarnorm{1/2}{v}^2
+\frac{C_\bopK^2c^2}{2\alpha_\bopW}\Ltwonorm{\overline{v}}^2
-\frac{c}2\Ltwonorm{\overline{v}}^2.
\label{diricostep3}
\end{align}
We also have that
\begin{align}\label{diricostep4}
\frac{c}2\left\langle v,\overline{v}\right\rangle
&=\frac{c}2\Ltwonorm{\overline{v}}^2
\end{align}

Taking $\alpha=\min(\alpha_\bopV,\alpha_\bopK)$ and $C=\max(C_\bopV,C_\bopK)$, and
putting \cref{diricostep1,diricostep2,diricostep3,diricostep4} together, we obtain
\begin{align*}
L
&\geqslant
\frac{\alpha}2\Hnorm{-1/2}{\mu}^2
+\frac{\alpha}2\Hstarnorm{1/2}{v}^2
+\left(
c-\frac{c^2C^2}{\alpha}
\right)\Ltwonorm{\overline{v}}^2.
\end{align*}
Letting $\displaystyle c=\frac{\alpha}{2C^2}$ gives
\begin{align*}
L
&\geqslant
\frac{\alpha}2\Hnorm{-1/2}{\mu}^2
+\frac{\alpha}2\Hstarnorm{1/2}{v}^2
+\frac{\alpha}{4C^2}\Ltwonorm{\overline{v}}^2\\
&\gtrsim
\Hnorm{-1/2}{\mu}^2
+\Hstarnorm{1/2}{v}^2
+\Ltwonorm{\overline{v}}^2.
\end{align*}

Finally, we show that
\begin{align*}
\Vnorm{(v,\mu)}
&=\Hnorm{1/2}{v}+\Hnorm{-1/2}{\mu}\\
&\leqslant\Hnorm{1/2}{v-\overline{v}}+\Hnorm{1/2}{\overline{v}}+\Hnorm{-1/2}{\mu}\\
&=\Hstarnorm{1/2}{v}+\Ltwonorm{\overline{v}}+\Hnorm{-1/2}{\mu},\\
\Vnorm{(w,\eta)}
&\leqslant\Hstarnorm{1/2}{v}+\Ltwonorm{\overline{v}}+\Hnorm{-1/2}{\mu+c\overline{v}}\\
&\leqslant\Hstarnorm{1/2}{v}+\Ltwonorm{\overline{v}}+\Hnorm{-1/2}{\mu}+c\Hnorm{-1/2}{\overline{v}}\\
&\lesssim\Hstarnorm{1/2}{v}+\Ltwonorm{\overline{v}}+\Hnorm{-1/2}{\mu}.
\end{align*}
Therefore
\begin{align*}
\Vnorm{(v,\mu)}\Vnorm{(w,\eta)}
&\lesssim
\Hnorm{-1/2}{\mu}^2
+\Hstarnorm{1/2}{v}^2
+\Ltwonorm{\overline{v}}^2\\
&\lesssim L.
\end{align*}
We obtain the first part of \cref{a:coercive} by dividing through by $\Vnorm{(w,\eta)}$ and taking the supremum.

To show the second part of \cref{a:coercive}, we let $(w,\eta)\in\productspace{W}\setminus\{0\}$ and proceed as follows.
\begin{align*}
L&:=\sup_{(v,\mu) \in \productspace{W}}\left|\form{A}[(v,\mu),(w,\eta)] + \form{B}\D[(v,\mu),(w,\eta)]\right|\\
&\geqslant
\form{A}[(w,\eta-\overline{w}),(w,\eta)] + \form{B}\D[(w,\eta-\overline{w}),(w,\eta)]\\
&=
-\langle\bopKadj\overline{w},w\rangle_\Gamma
+\langle\bopV\eta,\eta\rangle_\Gamma
-\langle\bopV\overline{w},\eta\rangle_\Gamma
+\langle\bopW w,w\rangle_\Gamma
+\tfrac12\langle\overline{w},w\rangle_\Gamma
.
\end{align*}
This is of the same form as \cref{eq:dirichlet_zero_first_bit}, so we proceed as above to obtain
\[
L\gtrsim\Vnorm{(v,\mu)}\Vnorm{(w,\eta)}.
\]
This is greater than zero for all $(w,\eta)\not=0$, and so we have proven the second part of \cref{a:coercive}.

\Cref{a:discoercive} can be proven in the same way as above using the discrete space $\productspace{V}_h$ in the place of $\productspace{W}$.
\end{proof}

\begin{proposition}[Continuity]\label{prop:diri_continuity}
\Cref{a:continuity} is satisfied for the Dirichlet problem.
\end{proposition}
\begin{proof}
Applying \cref{lemma:cont_cald}, the relation
$$\left\langle \eta,v \right\rangle_\Gamma \leqslant \Hnorm{-1/2}{\eta} \Hnorm{1/2}{v},$$
and the Cauchy--Schwarz inequality,
$$\betaparam\D \left\langle  w,v \right\rangle_{\Gamma} \leqslant \betaparam\D^{1/2} \Ltwonorm{w}
\betaparam\D^{1/2} \Ltwonorm{v},$$
to the form
$\form{A}+\form{B}\D$ yields the desired continuity result.
\end{proof}

\begin{proposition}[Approximation]
\Cref{a:approx} is satisfied for the Dirichlet problem if $0\leqslant\betaparam\D\lesssim h^{-1}$.
\end{proposition}
\begin{proof}
Using standard approximation results 
(see eg \cite[theorems 10.4 and 10.9]{Stein07}), we see that
\begin{align*}
\inf_{(w_h,\eta_h)\in\productspace{V}_h}\Vnorm{(v-w_h,\mu-\eta_h)}
&=\inf_{w_h\in \Vee^k_h}\Hnorm{1/2}{v-w_h}+\inf_{\eta_h\in\Lambda^l_h}\Hnorm{-1/2}{\mu-\eta_h}\\
&\lesssim h^{\zeta-1/2}\Hseminorm{\zeta}{v}
         + h^{\xi+1/2}\Hseminorm{\xi}{\mu},\\
\inf_{w_h\in\Vee^k_h}\Ltwonorm[\Gamma\D]{v-w_h}&\lesssim h^{\zeta}\Hseminorm{\zeta}{v}.
\end{align*}
Applying these to the definition of $\starnorm{\cdot}$ gives
\begin{equation*}
\inf_{(w_h,\eta_h)\in\productspace{V}_h}\starnorm{(v-w_h,\mu-\eta_h)}
\lesssim h^{\zeta-1/2}\Hseminorm{\zeta}{v}
         + h^{\xi+1/2}\Hseminorm{\xi}{\mu}
         + \betaparam\D^{1/2}h^{\zeta}\Hseminorm{\zeta}{v}.
\end{equation*}
If $\betaparam\D=0$, \cref{a:approx} holds.
If $0<\betaparam\D\lesssim h^{-1}$, then $\betaparam\D^{1/2}h^{\zeta}\lesssim h^{\zeta-1/2}$, and so \cref{a:approx} holds.
\end{proof}

We have shown that \cref{a:coercive,a:discoercive,a:continuity,a:approx} are satisfied. Additionally the extra assumptions in \cref{prop:exist_uniqueness}
are satisfied, so we conclude that the results of
\cref{prop:exist_uniqueness,prop:best_approximation,col:error_est,prop:bulk_error,cor:bulk_est} apply to the Dirichlet problem.
This is summarised in the following result.

\begin{theorem}\label{main_dirichlet_result}
The Dirichlet problem \cref{eq:abstract_form_Dir} has a unique solution $(u,\lambda)\in H^{s}(\Gamma)\times H^{r}(\Gamma)$, for some $s\geqslant\tfrac12$ and $r\geqslant-\tfrac12$.
The discrete Dirichlet problem \cref{BEMform_Dir} is invertible. If
$\exists\betamin>0$ such that $\betamin<\betaparam\D\lesssim h^{-1}$
or $\betaparam\D=0$, its solution $(u_h,\lambda_h)\in\Vee_h^k\times\Lambda_h^l$ satisfies
\[
\Bnorm[\D]{(u-u_h,\lambda-\lambda_h)} \lesssim
h^{\zeta-1/2} \Hseminorm{\zeta}{u} + h^{\xi+1/2}
\Hseminorm{\xi}{\lambda},
\]
where $\zeta = \min(k+1,s)$ and $\xi = \min(l+1,r)$. Additionally,
\[
\Hnorm[\Omega]{1}{\tilde u-\tilde u_h} \lesssim
h^{\zeta-1/2} \Hseminorm{\zeta}{u} + h^{\xi+1/2}
\Hseminorm{\xi}{\lambda},
\]
where $\tilde u$ and $\tilde u_h$ are the solutions in $\Omega$ computed using \cref{eq:represent}.
\end{theorem}

\subsection{Application of the theory to the Neumann problem}\label{sec:bem_neumann}
The Neumann problem takes the form: Find
$(u_h,\lambda_h) \in\zerostar{\productspace{V}}_h$ such that
\begin{align}\label{BEMform_Neu}
\form{A}[(u_h,\lambda_h),(v_h,\mu_h)] +
\form{B}\N[(u_h,\lambda_h),(v_h,\mu_h)] &= \form{L}\N(v_h,\mu_h)
&&\forall(v_h,\mu_h) \in \zerostar{\productspace{V}}_h.
\end{align}
Here $\zerostar{\productspace{V}}_h:=\zerostar{\Vee}^k_h(\Gamma)\times\Lambda^l_h(\Gamma)$
and $\zerostar{\Vee}^k_h(\Gamma):=\{v\in\Vee^k_h:\overline{v}=0\}$.

We introduce the following $\form{B}\N$-norm.
\[
\Bnorm[\N]{(v,\mu)} := \Vnorm{(v,\mu)} + \betaparam\N^{1/2} \Ltwonorm[\Gamma\N]{\mu},
\]
and we let $\starnorm{\cdot}=\Bnorm[\N]{\cdot}$.

We now proceed to verify that \cref{a:coercive,a:discoercive,a:continuity,a:approx} hold.

\begin{proposition}[Coercivity]
\Cref{a:coercive,a:discoercive} are satisfied for the Neumann problem with $\betaparam\N\geqslant0$.
\end{proposition}
\begin{proof}
As $v\in H_*^{1/2}(\Gamma\N)$, we may immediately apply \cref{lemma:coercivity1,lemma:coercivity2} to show that the form is coercive.
\end{proof}

\begin{proposition}[Continuity]
\Cref{a:continuity} is satisfied for the Neumann problem.
\end{proposition}
\begin{proof}
The proof is the same as in the Dirichlet case.
\end{proof}

\begin{proposition}[Approximation]
\Cref{a:approx} is satisfied for the Neumann problem if $0\leqslant\betaparam\N\lesssim h$.
\end{proposition}
\begin{proof}
The proof is the same as in the Dirichlet case.
\end{proof}

As in the Dirichlet case, 
the extra assumptions in \cref{prop:exist_uniqueness} are satisfied. We therefore conclude with the following result.
\begin{theorem}\label{main_neumann_result}
The Neumann problem \cref{eq:abstract_form_Neu} has a unique solution $(u,\lambda)\in H_*^{s}(\Gamma)\times H^{r}(\Gamma)$,
for some $s\geqslant\tfrac12$ and $r\geqslant0$ if $\betaparam\N>0$.
If $\betaparam\N=0$, this holds for some $r\geqslant-\tfrac12$.
The discrete Neumann problem \cref{BEMform_Neu} is invertible. If $0\leqslant\betaparam\N\lesssim h$, its solution $(u_h,\lambda_h)\in\zerostar{\Vee}_h^k\times\Lambda_h^l$ satisfies
\[
\Bnorm[\N]{(u-u_h,\lambda-\lambda_h)} \lesssim
h^{\zeta-1/2} \Hseminorm{\zeta}{u} + h^{\xi+1/2}
\Hseminorm{\xi}{\lambda},
\]
where $\zeta = \min(k+1,s)$ and $\xi = \min(l+1,r)$. Additionally,
\[
\Hnorm[\Omega]{1}{\tilde u-\tilde u_h} \lesssim
h^{\zeta-1/2} \Hseminorm{\zeta}{u} + h^{\xi+1/2}
\Hseminorm{\xi}{\lambda},
\]
where $\tilde u$ and $\tilde u_h$ are the solutions in $\Omega$ computed using \cref{eq:represent}.
\end{theorem}


\subsection{Application of the theory to the mixed Dirichlet--Neumann problem}\label{sec:bem_mixed_dn}
For the mixed problem, the boundary element method takes the form: Find $(u_h,\lambda_h) \in
\productspace{V}_h$ such that
\begin{multline}\label{eq:Nitsche_mixed_BEM}
\form{A}[(u_h,\lambda_h),(v_h,\mu_h)]
+\form{B}\ND[(u_h,\lambda_h),(v_h,\mu_h)] = \form{L}\ND(v_h,\mu_h)
\\\forall (v_h,\mu_h) \in
\productspace{V}_h.
\end{multline}

We now show that the assumptions for the abstract error estimate are
satisfied for the formulation \cref{eq:Nitsche_mixed_BEM}. First, we
introduce the following norms.
\begin{align*}
\Bnorm[\ND]{(v,\mu)} &:= \Vnorm{(v,\mu)} + \betaparam\D^{1/2} \Ltwonorm[\Gamma\D]{v} + \betaparam\N^{1/2} \Ltwonorm[\Gamma\N]{\mu}\\
\starnorm{(v,\mu)}&:=\Vnorm{(v,\mu)}+ \betaparam\D^{1/2} \Ltwonorm{v}+\betaparam\N^{1/2}\Ltwonorm{\mu}.
\end{align*}
Observe that in this case the two norms are not the same, nor are they
equivalent, so the below results cannot be used to prove existence of
a unique solution to \cref{eq:MND_compact}. Nevertheless, it is easy to verify
that if the exact solution to the mixed Dirichlet--Neumann problem is in $\productspace{V}$
then it satisfies \cref{eq:MND_compact}.

\begin{proposition}[Coercivity]
\Cref{a:coercive,a:discoercive} are satisfied for the mixed Dirichlet--Neumann problem if $\exists\betamin>0$,
independent of $h$, such that $\betaparam\D>\betamin$.
\end{proposition}
\begin{proof}
We obtain using \cref{lemma:coerciv_cald} that for $(v,\mu) \in \productspace{W}$,
\begin{align*}
L&:=
\form{A}[(v,\mu),(v,\mu)] +
\form{B}\ND[(v,\mu),(v,\mu)]
\\
&\geqslant
  \alpha\Hnorm{-1/2}{\mu}^2
+ \alpha\Hstarnorm{1/2}{v}^2
+ \betaparam\D \Ltwonorm[\Gamma\D]{v}^2
+ \betaparam\N \Ltwonorm[\Gamma\N]{\mu}^2.
\end{align*}
Taking $\alpha'=\min(\alpha,\betamin/2)$, we get
\begin{multline*}
L\geqslant
  \alpha'\Hnorm{-1/2}{\mu}^2
+ \alpha'\left(\Hstarnorm{1/2}{v}^2+\Ltwonorm[\Gamma\D]{v}^2\right)
\\+ (\betaparam\D-\alpha') \Ltwonorm[\Gamma\D]{v}^2
+ \betaparam\N \Ltwonorm[\Gamma\N]{\mu}^2.
\end{multline*}
By \cite[theorem 2.6]{Stein07},
$\left(\Hstarnorm{1/2}{\cdot}^2+\Ltwonorm[\Gamma\D]{\cdot}^2\right)^{1/2}$ is an equivalent norm to
$\Hnorm{1/2}{\cdot}$. Therefore
\begin{align*}
L&\geqslant
  \alpha'\Hnorm{-1/2}{\mu}^2
+ \alpha'\Hnorm{1/2}{v}^2
+ \betaparam\D\left(1-\frac{\alpha'}{\betamin}\right) \Ltwonorm[\Gamma\D]{v}^2
+ \betaparam\N \Ltwonorm[\Gamma\N]{\mu}^2
\\
&\gtrsim
  \Hnorm{-1/2}{\mu}^2
+ \Hnorm{1/2}{v}^2
+ \betaparam\D \Ltwonorm[\Gamma\D]{v}^2
+ \betaparam\N \Ltwonorm[\Gamma\N]{\mu}^2
\end{align*}
Coercivity follows using the definition of $\Bnorm[\ND]{\cdot}$.
\end{proof}

\begin{proposition}[Continuity]
\Cref{a:continuity} is satisfied for the mixed Dirichlet--Neumann problem if
$\exists\betamin>0$, independent of $h$, such that $\betaparam\D^{1/2}\betaparam\N^{1/2}>\betamin$.
\end{proposition}
\begin{proof}
Using the fact that $\left\langle v,\mu\right\rangle_\Gamma=\left\langle v,\mu\right\rangle_{\Gamma\D}+\left\langle v,\mu\right\rangle_{\Gamma\N}$,
we see that
\begin{align*}
\form{B}\ND[(w,\eta),(v,\mu)]\vspace{-200mm}&=
\tfrac12\left\langle w,\mu\right\rangle_{\Gamma\D}
-\tfrac12\left\langle\eta,v\right\rangle_{\Gamma\D}
+\betaparam\D\left\langle w,v\right\rangle_{\Gamma\D}
\\&\qquad\qquad\qquad
+\tfrac12\left\langle\eta,v\right\rangle_{\Gamma\N}
-\tfrac12\left\langle w,\mu\right\rangle_{\Gamma\N}
+\betaparam\N\left\langle\eta,\mu\right\rangle_{\Gamma\N}\\
&=
\tfrac12\left\langle w,\mu\right\rangle_\Gamma
-\left\langle\eta,v\right\rangle_{\Gamma\D}
+\betaparam\D\left\langle w,v\right\rangle_{\Gamma\D}
\\&\qquad\qquad\qquad\quad
+\tfrac12\left\langle\eta,v\right\rangle_\Gamma
-\left\langle w,\mu\right\rangle_{\Gamma\N}
+\betaparam\N\left\langle\eta,\mu\right\rangle_{\Gamma\N}\\
&\lesssim
\tfrac12\left\langle w,\mu\right\rangle_\Gamma
-\betaparam\D^{1/2}\betaparam\N^{1/2}\left\langle\eta,v\right\rangle_{\Gamma\D}
+\betaparam\D\left\langle w,v\right\rangle_{\Gamma\D}
\\&\qquad\qquad\qquad\quad
+\tfrac12\left\langle\eta,v\right\rangle_\Gamma
-\betaparam\D^{1/2}\betaparam\N^{1/2}\left\langle w,\mu\right\rangle_{\Gamma\N}
+\betaparam\N\left\langle\eta,\mu\right\rangle_{\Gamma\N}.
\end{align*}
Proceeding as in \cref{prop:diri_continuity} leads to the desired result.
\end{proof}

\begin{proposition}[Approximation]\label{dn:approx}
\Cref{a:approx} is satisfied for the mixed Dirichlet--Neumann problem if $0<\betaparam\D\lesssim h^{-1}$ and $0<\betaparam\N\lesssim h$.
\end{proposition}
\begin{proof}
Proceeding as in the Dirichlet case, we see that
\begin{multline*}
\inf_{(w_h,\eta_h)\in\productspace{V}_h}\starnorm{(v-w_h,\mu-\eta_h)}
\lesssim h^{\zeta-1/2}\Hseminorm{\zeta}{v} + h^{\xi+1/2}\Hseminorm{\xi}{\mu}\\
         + \betaparam\D^{1/2}h^{\zeta}\Hseminorm{\zeta}{v}
         + \betaparam\N^{1/2}h^{\xi}\Hseminorm{\xi}{\mu}
\end{multline*}
If $0<\betaparam\D\lesssim h^{-1}$ and $0<\betaparam\N\lesssim h$, then
\begin{equation*}
\betaparam\D^{1/2}h^{\zeta}\Hseminorm{\zeta}{v} + \betaparam\N^{1/2}h^{\xi}\Hseminorm{\xi}{\mu}
\lesssim
h^{\zeta-1/2}\Hseminorm{\zeta}{v} + h^{\xi+1/2}\Hseminorm{\xi}{\mu},
\end{equation*}
and so \cref{a:approx} holds.
\end{proof}

Motivated by the bounds on $\betaparam\D$ and $\betaparam\N$ in this proposition, we will later take
$\betaparam\D=\betaparam h^{-1}$
and
$\betaparam\N=\betaparam h$, where $\betaparam$ is a constant.

If $k=l$, $\betaparam\N\lesssim h^{-1}$, and the solution is smooth enough, then
\begin{equation*}
\betaparam\N^{1/2}h^{\xi}=
\betaparam\N^{1/2}h^{\zeta}\lesssim h^{\zeta-1/2}.
\end{equation*}
Therefore the same order of convergence will be observed when the bounds on $\betaparam\N$ here and in the
theorem below may be replaced by $\betaparam\N\lesssim h^{-1}$ without loss of convergence. In this case, both $\betaparam\N$ and $\betaparam\D$
may be taken to be constants independent of $h$.

We conclude that the best approximation result of \cref{prop:best_approximation}
and the error estimate of \cref{col:error_est} hold for the discrete
solutions of \cref{eq:Nitsche_mixed_BEM}, as given in the following theorem.

\begin{theorem}\label{main_dn_result}
Let $(u,\lambda)\in H^{s}(\Gamma)\times H^{r}(\Gamma)$, for some $s\geqslant\tfrac12$ and $r\geqslant0$,
be the unique solution to the mixed Dirichlet--Neumann problem. This solution satisfies \cref{eq:MND_compact}.
Let $(u_h,\lambda_h)\in\Vee_h^k\times\Lambda_h^l$ be the solution of \cref{eq:Nitsche_mixed_BEM}.
If $0<\betaparam\D\lesssim h^{-1}$, $0<\betaparam\N\lesssim h$ and 
$\exists\betamin>0$ such that $\betaparam\D^{1/2}\betaparam\N^{1/2}>\betamin$ and $\betaparam\D>\betamin$, then
\[
\Bnorm[\ND]{(u-u_h,\lambda-\lambda_h)} \lesssim
h^{\zeta-1/2} \Hseminorm{\zeta}{u} + h^{\xi+1/2}
\Hseminorm{\xi}{\lambda},
\]
where $\zeta = \min(k+1,s)$ and $\xi = \min(l+1,r)$.
\end{theorem}

If we set $\betaparam\D=0$ and $\betaparam\N=0$, we arrive at a penalty-free formulation for the mixed Dirichlet--Neumann problem.
We conjecture based on numerical experiments that this result also holds for the penalty-free formulation. The analysis for this case would take
a similar form as in the Dirichlet and Neumann penalty-free cases.

\subsection{Application of the theory to the Robin problem}\label{sec:bem_robin}
The formulation for Robin conditions was proposed in \cref{eq:cald_Rob_compact}. 
To simplify the notation we introduce a function $\robindenom:\Gamma \to \RR_+$ defined by
\[
\robindenom(\x) := \frac{1}{\param(\x) \betaparam\R(\x) + 1},
\]
and we assume that $\param$ and $\betaparam\R$ are sufficiently regular so that 
\begin{equation}\label{eq:phi_smooth}
\robindenom \in W^{1,2}(\Gamma) \cap \LL^{\infty}(\Gamma).
\end{equation}
This will be true if the mesh has some local quasi-uniformity and
$\param$ is smooth enough. Noting that
\[
\robindenom-\tfrac12=\frac{2-(\param\betaparam\R+1)}{2(\param\betaparam\R+1)}=-\tfrac{1}{2}\frac{\param\betaparam\R-1}{\param\betaparam\R+1},
\]
we may then write the operators $\form{B}\R$ and $\form{L}\R$ as
\begin{align}\label{eq:BR}
\form{B}\R[(u,\lambda),(v,\mu)]&=
\left\langle (\robindenom-\tfrac12) u,\mu \right\rangle_{\Gamma\R}
-\left\langle (\robindenom-\tfrac12) \lambda, v \right\rangle_{\Gamma\R}
+\left\langle \robindenom \betaparam\R u, v \right\rangle_{\Gamma\R}
+\left\langle  \robindenom \param\lambda, \mu \right\rangle_{\Gamma\R},\\
\form{L}\R[(v,\mu)]&=
\left\langle(g\D +\param g\N)\robindenom,\betaparam\R v + \mu \right\rangle_{\Gamma\R}.
\end{align}
The boundary element method for the Robin problem reads: Find $(u_h,\lambda_h)
\in \productspace{V}_h$ such that
\begin{align}\label{eq:Robin_BEM}
\form{A}[(u_h,\lambda_h),(v_h,\mu_h)] +\form{B}\R[(u_h,\lambda_h),(v_h,\mu_h)]
&=
\form{L}\R[(v_h,\mu_h)]
&&\forall (v_h,\mu_h) \in \productspace{V}_h.
\end{align}

For the analysis the following technical lemmas will be useful.
\begin{lemma}\label{product_fracspace}
If $\varphi \in W^{1,2}(\Gamma) \cap \LL^{\infty}(\Gamma)$ and $f \in
H^{1/2}(\Gamma)$, then $\varphi f \in H^{1/2}(\Gamma)$ and
\[
\Hnorm{1/2}{\varphi f } \leqslant
  C\left(\norm{\varphi}_{\LL^\infty(\Gamma)} + \norm{\varphi}_{W^{1,2}(\Gamma)}\right) \Hnorm{1/2}{f}.
\]
\end{lemma}
\begin{proof}
The proof is a consequence of \cite[lemma 6]{BM01} which shows that
\begin{equation}\label{Gagliardo_Nirenberg}
\Hnorm{1/2}{\varphi f } \leqslant
  C\left(\norm{\varphi}_{\LL^\infty(\Gamma)} \Hnorm{1/2}{f} +
    \norm{f}_{\LL^4(\Gamma)} \norm{\varphi}_{W^{1,2}(\Gamma)}^{1/2} \norm{\varphi}_{\LL^\infty(\Gamma)}^{1/2}\right).
\end{equation}
We then recall the Sobolev injection $\norm{f}_{\LL^4(\Gamma)} \leqslant C
\Hnorm{1/2}{f}$ from \cite[theorem 6.7]{NPV12} and
conclude using this result and an arithmetic-geometric inequality of
the right hand side of \cref{Gagliardo_Nirenberg}.
\end{proof}

\begin{lemma}\label{product_coercive}
If $\varphi,f\in\LL^{2}(\Gamma)$ and $\varphi(\x)>0$ for all $\x\in\Gamma$, then there exists $C>0$ such that
\[
\Ltwonorm{\varphi f}^2 \geqslant
  C\Ltwonorm{f}^2.
\]
\end{lemma}
\begin{proof}
Let $a=\inf_{\x\in\Gamma}\varphi(\x)$. Since $\Gamma$ is closed, there exists $\y\in\Gamma$ such that $\varphi(\y)=a$.
Therefore $a>0$. We now see that
\begin{align*}
\Ltwonorm{\varphi f}^2&=\int_\Gamma \varphi^2f^2\\
&\geqslant a^2\int_\Gamma f^2\\
&= C\Ltwonorm{f}^2,
\end{align*}
where $C = a^2$.
\end{proof}

We introduce the norm
\[
\Bnorm[\R]{(v,\mu)} :=\Vnorm{(v,\mu)} +
\Ltwonorm{(\param \robindenom)^{1/2}\mu}+\Ltwonorm{(\robindenom \betaparam\R)^{1/2} v}
\]
and set $\starnorm{\cdot}=\Bnorm[\R]{\cdot}$.
We note that if $\param\to0$ or $\param\to\infty$, then $\Bnorm[\R]\cdot$ converges to $\Bnorm[\D]\cdot$ or $\Bnorm[\N]\cdot$ respectively.
We now proceed to show that \cref{a:coercive,a:discoercive,a:continuity,a:approx} hold.

\begin{proposition}[Coercivity]
\Cref{a:coercive,a:discoercive} are satisfied for the Robin problem.
\end{proposition}
\begin{proof}
Let $(v,\mu) \in \productspace{W}$, and let 
$L:=\form{A}[(v,\mu),(v,\mu)] +\form{B}\R[(v,\mu),(v,\mu)]$.
Using \cref{lemma:coerciv_cald}, we see that
\begin{multline*}
L\geqslant
\alpha\Hnorm{-1/2}{\mu}^2
+\alpha\Hnorm{1/2}{v}^2
-\alpha\Ltwonorm{v}^2
\\+\Ltwonorm{(\param\robindenom)^{1/2}\mu}^2
+\Ltwonorm{(\robindenom\betaparam\R)^{1/2}v}^2,
\end{multline*}
for any $\alpha\leqslant\min(\alpha_\bopV,\alpha_\bopW)$.

By \cref{product_coercive}, we have
\begin{equation}
-\alpha\Ltwonorm{v}^2\geqslant-\frac{\alpha}C\Ltwonorm{(\robindenom\betaparam\R)^{1/2}v}^2.
\end{equation}
Taking $\alpha=\min(\alpha_\bopV,\alpha_\bopW,C/2)$, we obtain
\begin{equation*}
L\geqslant
\alpha\Hnorm{-1/2}{\mu}^2
+\alpha\Hnorm{1/2}{v}^2
+\Ltwonorm{(\param\robindenom)^{1/2}\mu}^2
+\tfrac12\Ltwonorm{(\robindenom\betaparam\R)^{1/2}v}^2,
\end{equation*}
Using the definition of $\Bnorm[\R]{\cdot}$, we see that the form is coercive.
\end{proof}

\begin{proposition}[Continuity]\label{Robincont}
\Cref{a:continuity} is satisfied for the Robin problem if $\exists\betamin>0$, independent of $h$, such that
$\betaparam\R>\betamin$.
\end{proposition}
\begin{proof}
Using \cref{product_fracspace}, we see that for $g \in
H^{-1/2}(\Gamma)$, $\varphi \in W^{1,2}(\Gamma) \cap
\LL^\infty(\Gamma)$, and $f \in H^{1/2}(\Gamma)$, 
\[
\left\langle  \robindenom g, f \right\rangle_\Gamma \leqslant  C\left(\norm{\varphi}_{\LL^\infty(\Gamma)} + \norm{\varphi}_{W^{1,2}(\Gamma)}\right) \Hnorm{-1/2}{g}\Hnorm{1/2}{f}.
\]
Let $\param_{\min}:=\inf_{\x\in\Gamma}\param(\x)$. As in the proof of \cref{product_coercive}, we see that $\param_{\min}>0$.
Hence,
\begin{equation*}
-\tfrac12<\robindenom-\tfrac12<\frac1{\betamin\param_{\min}+1},
\end{equation*}
and so
\begin{align*}
\norm{\robindenom-\tfrac12}_{\LL^\infty(\Gamma)} + \norm{\robindenom-\tfrac12}_{W^{1,2}(\Gamma)}
&<
\max\left(\tfrac12,\frac1{\betamin\param_{\min}+1}\right)
\left(\norm{1}_{\LL^\infty(\Gamma)} + \norm{1}_{W^{1,2}(\Gamma)}\right).
\end{align*}
Applying these two results to the first two boundary terms in $\form{B}\R[(w,\eta),(v,\mu)]$, we obtain
\begin{align*}
\left\langle  (\robindenom-\tfrac12) w, \mu \right\rangle_\Gamma-\left\langle (\robindenom-\tfrac12) v,\eta \right\rangle_\Gamma
&\leqslant
C\Vnorm{(w,\eta)} \Vnorm{( v,\mu)}.
\end{align*}
By the Cauchy--Schwarz inequality, we obtain for the remaining terms
\begin{multline*}
\left\langle \robindenom \param \eta, \mu \right\rangle_{\Gamma} +
\left\langle  \robindenom \betaparam\R w, v \right\rangle_{\Gamma}
\\\leqslant \Ltwonorm{(\robindenom \param)^{1/2} \eta}\Ltwonorm{(\robindenom\param)^{1/2} \mu}
  + \Ltwonorm{(\robindenom \betaparam\R)^{1/2} w} \Ltwonorm{(\robindenom \betaparam\R)^{1/2}  v}.
\end{multline*}
Collecting the terms, we then have
\[
\form{B}\R[(w,\eta),(v,\mu)] \lesssim \Bnorm[\R]{(w,\eta)} \Bnorm[\R]{(v,\mu)}.
\]
\end{proof}

\begin{proposition}[Approximation]
\Cref{a:approx} is satisfied for the Robin problem if $\betaparam\R\eqsim h^{-1}$.
\end{proposition}
\begin{proof}
First note that $\robindenom<1$ and
$$\robindenom\param=\frac{\param}{\param\betaparam\R+1}=\frac{1}{\betaparam\R+\frac1\param}<\frac1{\betaparam\R}.$$
Therefore, 
\begin{align}\label{eq:robin_term_bounds}
\Ltwonorm{(\robindenom\betaparam\R)^{1/2}v}&\leqslant\betaparam\R^{1/2}\Ltwonorm{v}&&\text{and}&
\Ltwonorm{(\robindenom\param)^{1/2}\mu}&\leqslant\betaparam\R^{-1/2}\Ltwonorm{\mu}.
\end{align}
If $\betaparam\R\eqsim h^{-1}$, then \cref{a:approx} can be shown to hold.
\end{proof}

When using equal order approximation, the same order of convergence
will be observed when the bounds on $\betaparam\R$ here and in the theorem below may be replaced by $h\lesssim\betaparam\R\lesssim h^{-1}$
for sufficiently smooth solutions. Note that the condition $h^{-1}\lesssim\betaparam\R$ implies the existance of $\betamin$, as required
by \cref{Robincont}. The condition $h\lesssim\betaparam\R$ does not imply this, so in this case the additional requirement that $\exists\betamin>0$
such that $\betamin<\betaparam\R$ is necessary to ensure continuity.

\begin{proposition}
The extra assumptions in \cref{prop:exist_uniqueness} are satisfied for the Robin problem.
\end{proposition}
\begin{proof}
As a consequence of the coercivity and continuity above and observing
that by the Cauchy--Schwarz inequality and the definition of $\robindenom$,
there exists $C$ such that
\[
\left\langle \robindenom(g\D+ \param g\N), \betaparam\R v + \mu \right\rangle_{\Gamma} \leqslant
C(\Ltwonorm{g\D} + \Ltwonorm{g\N}) \Bnorm[\R]{(v,\mu)}
\]
\end{proof}

We conclude that \cref{prop:exist_uniqueness,prop:best_approximation,col:error_est,cor:bulk_est} hold for the Robin problem.
This is summarised in the following result.

\begin{theorem}\label{main_robin_result}
The Robin problem \cref{eq:cald_Rob_compact} has a unique solution $(u,\lambda)\in H^{s}(\Gamma)\times H^{r}(\Gamma)$, for some $s\geqslant\tfrac12$ and $r\geqslant0$.
The discrete Robin problem \cref{eq:Robin_BEM} is invertible. If $\betaparam\R\eqsim h^{-1}$, its solution $(u_h,\lambda_h)\in\Vee_h^k\times\Lambda_h^l$ satisfies
\[
\Bnorm[\R]{(u-u_h,\lambda-\lambda_h)} \leqslant C\left(
h^{\zeta-1/2} \Hseminorm{\zeta}{u} + h^{\xi+1/2}
\Hseminorm{\xi}{\lambda} \right),
\]
for some $C>0$,
where $\zeta = \min(k+1,s)$ and $\xi = \min(l+1,r)$. Additionally,
\[
\Hnorm[\Omega]{1}{\tilde u-\tilde u_h} \leqslant C\left(
h^{\zeta-1/2} \Hseminorm{\zeta}{u} + h^{\xi+1/2}
\Hseminorm{\xi}{\lambda} \right),
\]
where $\tilde u$ and $\tilde u_h$ are the solutions in $\Omega$ computed using \cref{eq:represent}.
\end{theorem}

Again, we could set $\betaparam\R=0$ to arrive at a penalty-free formulation for Robin problems.
In this case, our numerical experiments show large errors for some values of the parameter $\param$, which leads us to conclude that this result
does not hold for the penalty-free formulation.

As $\param\to0$ and $\param\to\infty$, we obtain the Dirichlet and Neumann formulations analysed in \cref{sec:bem_dirichlet,sec:bem_neumann}.
We expect the condition number of the discrete system for the Robin problem to be no worse than in either extreme case, and observe this in
\cref{sec:robin_num}.


\section{Numerical results}\label{sec:numerical}
Drawing inspiration from \cite{JuSt09}, we define
\begin{align*}
u(x,y,z)&=\sin(\pi x)\sin(\pi y)\sinh(\sqrt2\pi z)\\
g\D(x,y,z)&=\sin(\pi x)\sin(\pi y)\sinh(\sqrt2\pi z),\\
g\N(x,y,z)&=\begin{pmatrix}
\pi\cos(\pi x)\sin(\pi y)\sinh(\sqrt2\pi z)\\
\pi\sin(\pi x)\cos(\pi y)\sinh(\sqrt2\pi z)\\
\sqrt2\pi\sin(\pi x)\sin(\pi y)\cosh(\sqrt2\pi z)
\end{pmatrix}\cdot\B\nu.
\end{align*}
It is easy to check that for any bounded domain $\Omega\subset\RR^3$ with boundary $\Gamma=\Gamma\D\cup\Gamma\N\cup\Gamma\R$
and any fixed $\param\in\RR$,
$u$ is the solution of
\begin{subequations}
\begin{align}
-\Delta u &=0       &&\text{in } \Omega,\\
u &= g\D              &&\text{on } \Gamma\D, \\
\frac{\partial u}{\partial\B\nu} &= g\N
                    &&\text{on } \Gamma\N,\\
\frac{\partial u}{\partial\B\nu} &= \frac{1}{\param} (u - g\D) + g\N
                    &&\text{on }\Gamma\R.
\end{align}\end{subequations}

In the examples presented here, we let $\Omega$ be the unit sphere, and $\Gamma$ its boundary.
In the computations presented, a series of approximations of the sphere by plane triangles are used.
The results in this section were computed using the boundary element library Bempp \cite{Bempp}, an open source boundary element library
developed by the authors of this paper.
All examples in this paper were computed with version 3.3.2 of the Bempp library.
Jupyter notebooks demonstrating the functionality used in this paper will be made available at \url{www.bempp.com}.

\subsection{Dirichlet boundary conditions}
First, we look at the case where $\Gamma=\Gamma\D$, in which the problem reduces to the Dirichlet problem:
\begin{subequations}
\begin{align}
-\Delta u &=0       &&\text{in } \Omega,\\
u &= g\D              &&\text{on } \Gamma.
\end{align}\end{subequations}

For this problem, we compare the penalty method proposed in this paper \cref{BEMform_Dir} to the standard single layer formulation:
Find $\lambda\in \Lambda_h$ such that
\begin{align}\label{usual_Dir}
\left\langle\bopV \lambda,\mu\right\rangle&=
\left\langle(\tfrac12\bopI+\bopK) g\D,\mu\right\rangle
&&\forall\mu\in \Lambda_h.
\end{align}

\newcommand{\errorlabel}[1]{$\Bnorm[#1]{(u-u_h,\lambda-\lambda_h)}$}

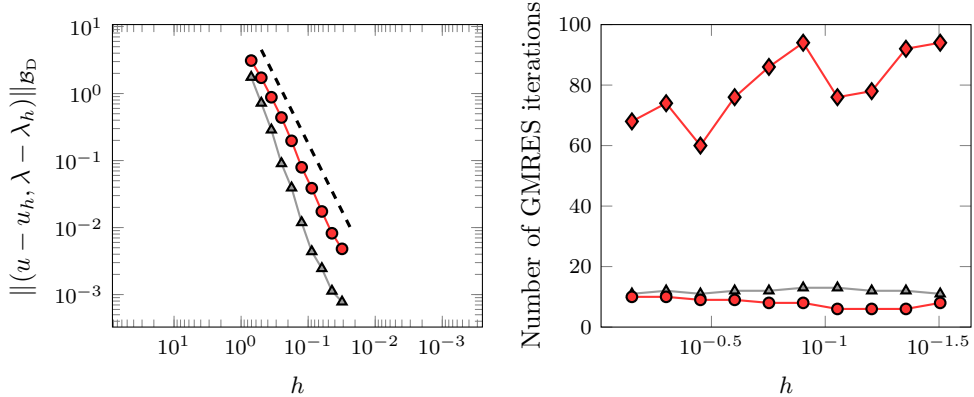
\begin{figure}
\centering
\begin{tikzpicture}
\begin{axis}[small,axis equal,axis on top,
             xmin=0.011,xmax=1.9,xlabel={$h$},x dir=reverse,xmode=log,
             ylabel near ticks,ylabel={\errorlabel{\D}},ymode=log
]
\referenceplot
table {%
0.5 4.5
0.0221 0.0087890625
};
\usualoneplot table {img/data/linlin/diri1usual_convergence.dat};
\Dlinlinplot table {img/data/linlin/diri1_convergence.dat};
\end{axis}
\end{tikzpicture}
\hfill
\begin{tikzpicture}
\begin{axis}[small,axis on top,
             xlabel={$h$},x dir=reverse,xmode=log,
             ylabel near ticks,ylabel={Number of GMRES iterations},ymin=0,ylabel near ticks,ymax=100
]
\usualoneplot table {img/data/linlin/diri1usual_iterations.dat};
\Dlinlinplot table {img/data/linlin/diri1_iterations.dat};
\Dtwoplot table {img/data/linlin/diri1nop_iterations.dat};
\end{axis}
\end{tikzpicture}
\caption{The convergence (left) and GMRES iteration counts (right) of the penalty method with $\betaparam\D=0.1$ (\Dlinlindesc)
compared to the standard single layer method \cref{usual_Dir} (\usualonedesc),
for the Dirichlet problem on the unit sphere, with $k=l=1$.
The iteration count plot shows the number of iterations taken to solve the mass matrix preconditioned system (\Dlinlindesc)
and the non-preconditioned system (\Dtwodesc).
The dashed line shows order 2 convergence.}
\label{fig:lin-diri1}
\end{figure}

\Cref{fig:lin-diri1} shows the convergence and iteration counts when $\betaparam\D=0.1$ and $k=l=1$, and so we look for
$(u_h,\lambda_h)\in\Vee_h^1\times\tilde\Lambda_h^1$.
We note that as $h$ decreases, $h^{-1}$ increases, so $0.1\lesssim h^{-1}$.
In this case, $\Gamma$ is smooth, and so $\Vee_h^1=\tilde\Lambda_h^1$.
The iteration count plot (right) shows the number of iterations taken to solve the non-preconditioned system (\Dtwodesc), compared with the system
with mass matrix preconditioned applied blockwise from the left (\Dlinlindesc), as described in \cite{Betcke17}.
Mass matrix preconditioning greatly reduces the number of iterations required, 
so for the remainder of this paper, we precondition all linear systems using mass matrix preconditioning.

For larger and more complex geometries, however, more specialised preconditioners are required.
With systems of boundary element equations, it is common to use operator preconditioning or Calder\'on preconditioning \cite{BC}, where
properties of the boundary operators at the continuous level are used to derive a preconditioned equation of a form known to be well conditioned.
In our case, it is not clear how to apply this approach, although further investigation of this warrants future work.

An alternative avenue of investigation leads to hierarchical LU based preconditioners, or even direct solvers of this type \cite{Bebend2005}.
The penalty terms in this paper are all sparse matrices that have non-zero entries only for neighbouring triangles, and so adding these
terms only affects the entries in the matrix arising from near interactions; the far interactions---which are exactly those that are approximated
in a hierarchical matrix compression---are not affected by these terms. Therefore H-matrix methodst can be applied to this method
with few algorithmic changes required.

\Cref{fig:lin-diriA} shows the dependence of the error and iteration count on the chosen value of $\betaparam\D$, for a range of values of $h$.
It can be seen that the number of iterations increases when $\betaparam\D$ is above around 0.1, and the error increases when $\betaparam\D$ is
above 100. This motivates our earlier choice of $0.1$ as the value of $\betaparam\D$, although anything smaller than this
appears to be a good choice of $\betaparam\D$.

In \cref{fig:lin-diri1}, it can be seen that the penalty method proposed here gives comparable convergence to the standard method in a similar
number of iterations. However, the system in the penalty method contains around twice the number of unknowns, and so each iteration will be more
expensive.

Additionally, the discrete systems for the penalty method are non-symmetric, so are solved using GMRES \cite{gmres}. The discrete systems for the
standard method \cref{usual_Dir} are symmetric, so CG \cite{cg} or MINRES \cite{minres} could be used: these methods are typically less expensive
than GMRES, so this is a further disadvantage of the penalty method for pure Dirichlet and Neumann problems and justifies our focus on more complex
boundary conditions.


\begin{figure}
\centering
\begin{tikzpicture}
\begin{axis}[small,axis on top,
    xlabel={$\betaparam\D$},xmin=1e-6*1.0001, xmax=1e6/1.0001,xmode=log,
    ylabel near ticks,ylabel={\errorlabel{\D}},ymax=100,ymode=log
]
\Doneplot table {img/data/linlin/diriA_convergence_4.dat};
\Dtwoplot table {img/data/linlin/diriA_convergence_7.dat};
\Dthreeplot table {img/data/linlin/diriA_convergence_10.dat};
\end{axis}
\end{tikzpicture}
\hfill
\begin{tikzpicture}

\begin{axis}[small,axis on top,
             xlabel={$\betaparam\D$},xmin=1e-6*1.0001, xmax=1e6/1.0001,xmode=log,
             ylabel near ticks,ylabel={Number of GMRES iterations},ymin=0,ymax=40
]
\Doneplot table {img/data/linlin/diriA_iterations_4.dat};
\Dtwoplot table {img/data/linlin/diriA_iterations_7.dat};
\Dthreeplot table {img/data/linlin/diriA_iterations_10.dat};
\end{axis}
\end{tikzpicture}
\caption{The dependence of the error (left) and iteration count (right) on the value of $\betaparam\D$ for
$h=2^{-2}$ (\Donedesc),
$h=2^{-3.5}$ (\Dtwodesc), and
$h=2^{-5}$ (\Dthreedesc),
for the Dirichlet problem on the unit sphere, with $k=l=1$.}
\label{fig:lin-diriA}
\end{figure}
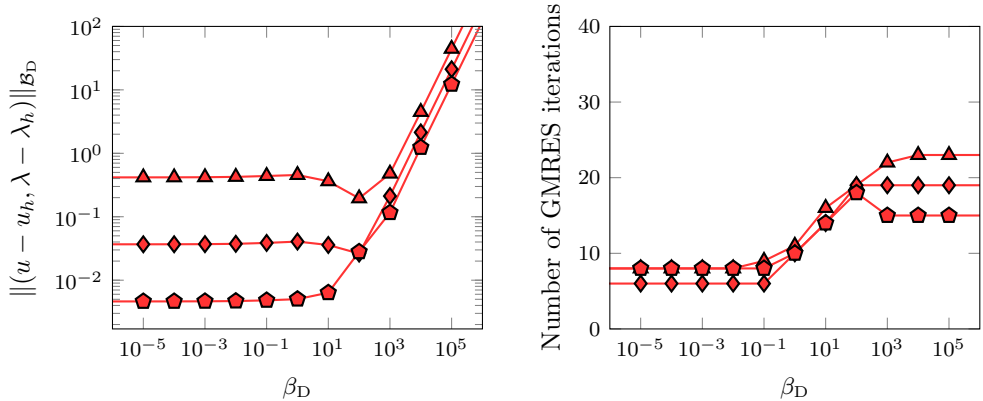

\subsection{Mixed Dirichlet--Neumann boundary conditions}
\label{sec:dn_numerical}
We now consider the case where $\Gamma=\Gamma\D\cup\Gamma\N$ and the problem reduces to a mixed Dirichlet--Neumann problem:
\begin{subequations}
\begin{align}
-\Delta u &=0       &&\text{in } \Omega,\\
u &= g\D              &&\text{on } \Gamma\D, \\
\frac{\partial u}{\partial\B\nu} &= g\N
                    &&\text{on } \Gamma\N.
\end{align}\end{subequations}
Let $\Gamma\N:=\{(x,y,z)\in\Gamma:x>0\}$ and $\Gamma\D:=\Gamma\setminus\Gamma\N$.
We use the same $g\D$ and $g\N$ as above.

We compare the method proposed in this paper with the standard method for mixed Dirichlet--Neumann problems \cite[equation (3.2)]{Stephan1987}:
Find $(u,\lambda)\in \tilde{H}^{1/2}(\Gamma\N)\times \tilde{H}^{-1/2}(\Gamma\D)$ such that
\begin{multline}\label{classicalformulation2}
\left\langle\bopW_{\text{N}\text{N}}u,v\right\rangle
+\left\langle\bopKadj_{\text{D}\text{N}},v\right\rangle
-\left\langle\bopK_{\text{N}\text{D}}u,\mu\right\rangle
+\left\langle\bopV_{\text{D}\text{D}}\lambda,\mu\right\rangle
\\=
-\left\langle\bopW_{\text{D}\text{N}}g\D,v\right\rangle
+\left\langle\left(\tfrac12\bopI-\bopKadj_{\text{N}\text{N}}\right)g\N,v\right\rangle
+\left\langle\left(\tfrac12\bopI+\bopK_{\text{D}\text{D}}\right)g\D,\mu\right\rangle
-\left\langle\bopV_{\text{N}\text{D}},\mu\right\rangle
\\
\forall(v,\mu)\in \tilde{H}^{1/2}(\Gamma\N)\times \tilde{H}^{-1/2}(\Gamma\D),
\end{multline}
where for a given boundary operator $\bop{B}$, $\bop{B}_{ij}$ is the corresponding boundary operator with the integral taken over $\Gamma_i$
and the point $\x\in\Gamma_j$. For example, $\bop{V}_{\text{N}\text{D}}$ is defined by
\begin{align}
[\bopV_{\text{N}\text{D}}f](\vec{x})&:=\int_{\Gamma\N}f(\vec{y})G(\x,\y)\dx[\vec{y}]
&&\text{for }\vec{x}\in\Gamma\D.
\end{align}

We first let $k=l+1=1$, and so look for 
$(u_h,\lambda_h)\in\Vee_h^1\times\Lambda_h^0$. As motivated above by \cref{dn:approx}, we set 
$\betaparam\D=\betaparam h^{-1}$
and
$\betaparam\N=\betaparam h$,
where $\betaparam$ is a constant.
The dependence of the error and iteration count on $\betaparam$ is shown in \cref{fig:con-neudiriA}.
We observe that $\betaparam=0.01$ is a good choice, as this gives a small error and iteration count.

The convergence of the error as we reduce $h$ is shown in \cref{fig:con-neudiri1}. Here we observe order 1.5 convergence, and
the same rate of convergence as the standard method \cref{classicalformulation2}, with a marginally lower error in the standard method.
The iteration count for the penalty method increases more gradually than the standard method, although this issue could be removed through better preconditioning of the standard
method.

\begin{figure}
\centering
\begin{tikzpicture}
\begin{axis}[small,axis on top,
             xlabel={$\betaparam$},xmin=1e-6*1.0001, xmax=1e6/1.0001,xmode=log,
             ylabel near ticks,ylabel={\errorlabel{\ND}},ymax=100,ymode=log
]
\NDoneplot table {img/data/conlin/neudiriA_convergence_4.dat};
\NDtwoplot table {img/data/conlin/neudiriA_convergence_7.dat};
\NDthreeplot table {img/data/conlin/neudiriA_convergence_10.dat};
\end{axis}
\end{tikzpicture}
\hfill
\begin{tikzpicture}

\begin{axis}[small,axis on top,
             xlabel={$\betaparam$},xmin=1e-6*1.0001, xmax=1e6/1.0001,xmode=log,
             ylabel near ticks,ylabel={Number of GMRES iterations},ymin=0,ymax=350
]
\NDoneplot table {img/data/conlin/neudiriA_iterations_4.dat};
\NDtwoplot table {img/data/conlin/neudiriA_iterations_7.dat};
\NDthreeplot table {img/data/conlin/neudiriA_iterations_10.dat};
\end{axis}
\end{tikzpicture}
\caption{The dependence of the error (left) and iteration count (right) on the value of $\betaparam$ for
$h=2^{-2}$ (\NDonedesc),
$h=2^{-3.5}$ (\NDtwodesc), and
$h=2^{-5}$ (\NDthreedesc),
for the mixed Dirichlet--Neumann problem on the unit sphere, with $k=l+1=1$.
Here we use $\betaparam\D=\betaparam h^{-1}$ and $\betaparam\N=\betaparam h$.
}
\label{fig:con-neudiriA}
\end{figure}
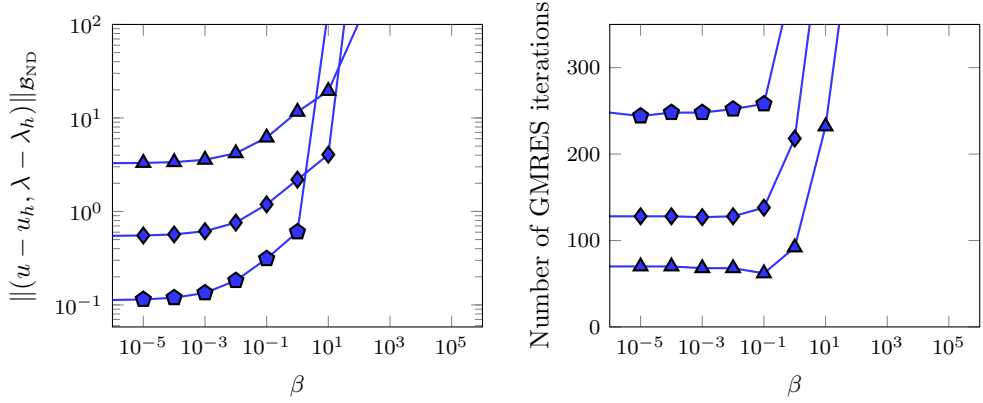

\begin{figure}
\centering
\begin{tikzpicture}
\begin{axis}[small,axis equal,axis on top,
             xlabel={$h$},xmode=log,x dir=reverse,
             ylabel near ticks,ylabel={\errorlabel{\ND}},ymode=log
]
\referenceplot table {%
0.4 32.925
0.02 0.3681126907959028
};
\usualoneplot table {img/mixed/conlin/hyp_convergence.dat};
\NDconlinplot table {img/data/conlin/neudiri1_convergence.dat};
\end{axis}
\end{tikzpicture}
\hfill
\begin{tikzpicture}
\begin{axis}[small,axis on top,
    xlabel={$h$},xmode=log,x dir=reverse,
    ylabel near ticks,ylabel={Number of GMRES iterations},ymin=0,ymax=350
]
\usualoneplot table {img/mixed/conlin/hyp_iterations.dat};
\NDconlinplot table {img/data/conlin/neudiri1_iterations.dat};
\end{axis}
\end{tikzpicture}
\caption{The convergence (left) and iterations counts (right) of the penalty method with $\betaparam=0.01$ (\NDconlindesc)
compared to the standard method \cref{classicalformulation2} (\usualonedesc),
for the mixed Dirichlet--Neumann problem on the unit sphere, with $k=l+1=1$.
The dashed line shows order 1.5 convergence.
Here we use $\betaparam\D=\betaparam h^{-1}$ and $\betaparam\N=\betaparam h$.
}
\label{fig:con-neudiri1}
\end{figure}
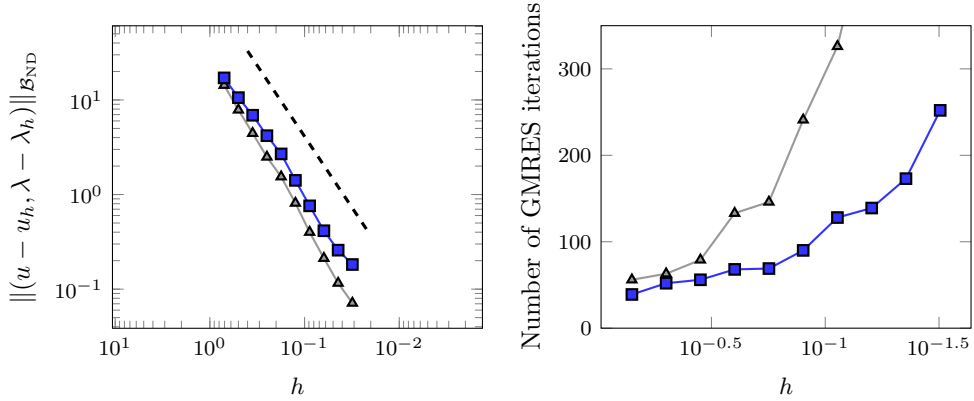

We next consider the case where $k=l=1$. In this case, as remarked in \cref{sec:bem_mixed_dn}, we may replace
the bound on $\betaparam\N$ by $\betaparam\N\lesssim h^{-1}$, and so we may take both $\betaparam\D$ and $\betaparam\N$ to be constant:
we set $\betaparam\D=\betaparam\N=\betaparam$. The dependence of the error and iteration count on $\betaparam$ for this choice of
parameters is shown in \cref{fig:lin-neudiriA}.


The convergence to the solution when $\betaparam=0.01$ is shown in \cref{fig:lin-neudiri1}.
It can be seen here that order 2 convergence is observed, higher than the expected order 1.5 convergence.
In this case, the standard method \cref{classicalformulation2} only achieves order 1 convergence, with a much higher iteration
count that the penalty method. For this choice of discrete spaces, we also compared the our method with the formulation
given in \cite[equation (1.19)]{Costabel1985}: this formulation is better conditioned than \cref{classicalformulation2} but still achieves
only order 1 convergence.

\begin{figure}
\centering
\begin{tikzpicture}
\begin{axis}[small,axis on top,
             xlabel={$\betaparam$},xmin=1e-6*1.0001, xmax=1e6/1.0001,xmode=log,
             ylabel near ticks,ylabel={\errorlabel{\ND}},ymax=10,ymode=log
]
\NDoneplot table {img/data/linlin/neudiriA_convergence_4.dat};
\NDtwoplot table {img/data/linlin/neudiriA_convergence_7.dat};
\NDthreeplot table {img/data/linlin/neudiriA_convergence_10.dat};
\end{axis}
\end{tikzpicture}
\hfill
\begin{tikzpicture}

\begin{axis}[small,axis on top,
             xlabel={$\betaparam$},xmin=1e-6*1.0001, xmax=1e6/1.0001,xmode=log,
             ylabel near ticks,ylabel={Number of GMRES iterations},ymin=0,ymax=100
]
\NDoneplot table {img/data/linlin/neudiriA_iterations_4.dat};
\NDtwoplot table {img/data/linlin/neudiriA_iterations_7.dat};
\NDthreeplot table {img/data/linlin/neudiriA_iterations_10.dat};
\end{axis}
\end{tikzpicture}
\caption{The dependence of the error (left) and iteration count (right) on the value of $\betaparam$ for
$h=2^{-2}$ (\NDonedesc),
$h=2^{-3.5}$ (\NDtwodesc), and
$h=2^{-5}$ (\NDthreedesc),
for the mixed Dirichlet--Neumann problem on the unit sphere, with $k=l=1$.
Here we use $\betaparam\D=\betaparam\N=\betaparam$.
}
\label{fig:lin-neudiriA}
\end{figure}
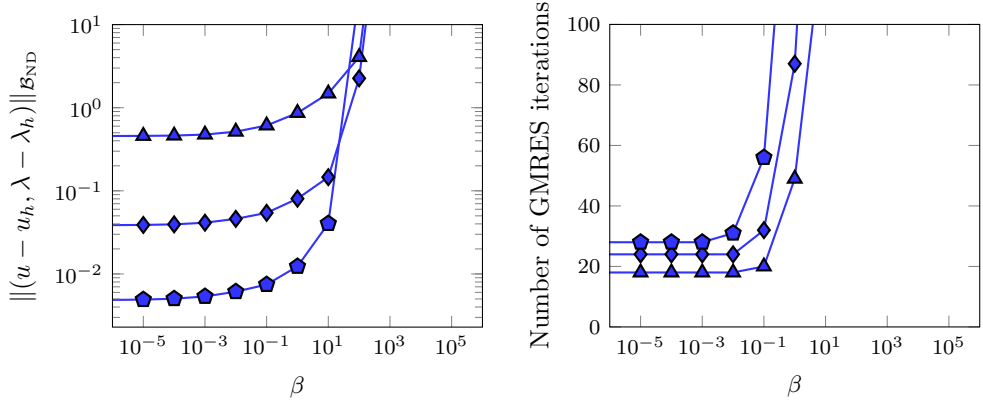

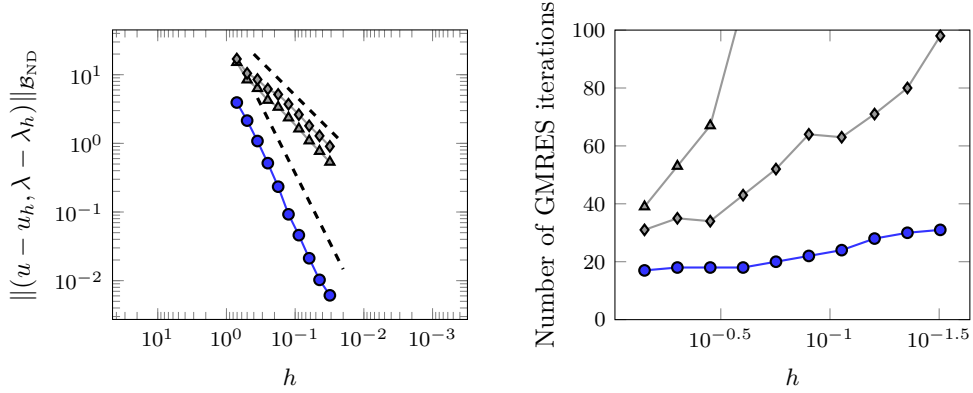
\begin{figure}
\centering
\begin{tikzpicture}
\begin{axis}[small,axis equal,axis on top,width=63mm,
             xlabel={$h$},xmode=log,x dir=reverse,
             ylabel near ticks,ylabel={\errorlabel{\ND}},ymode=log
]
\referenceplot table {%
0.4 20
0.02 1
};
\referenceplot table {%
0.35355 4.585
0.02 0.014672281414357528
};
\usualoneplot table {img/mixed/linlin/hyp_convergence.dat};
\usualtwoplot table {img/mixed/linlin/noh_convergence.dat};
\NDlinlinplot table {img/data/linlin/neudiri1_convergence.dat};
\end{axis}
\end{tikzpicture}
\hfill
\begin{tikzpicture}
\begin{axis}[small,axis on top,width=63mm,
    xlabel={$h$},xmode=log,x dir=reverse,
    ylabel near ticks,ylabel={Number of GMRES iterations},ymin=0,ymax=100
]
\usualoneplot table {img/mixed/linlin/hyp_iterations.dat};
\usualtwoplot table {img/mixed/linlin/noh_iterations.dat};
\NDlinlinplot table {img/data/linlin/neudiri1_iterations.dat};
\end{axis}
\end{tikzpicture}
\caption{The convergence (left) and iterations counts (right) of the penalty method with $\betaparam=0.01$ (\NDlinlindesc)
compared to the standard method \cref{classicalformulation2} (\usualonedesc) and the method given in \cite[equation (1.19)]{Costabel1985} (\usualtwodesc),
for the mixed Dirichlet--Neumann problem on the unit sphere, with $k=l=1$.
The dashed lines show order 2 and order 1 convergence.
Here we use $\betaparam\D=\betaparam\N=\betaparam$.
}
\label{fig:lin-neudiri1}
\end{figure}

In \cref{fig:con-neudiriA,fig:lin-neudiri1}, the error and iteration count remain steady as $\betaparam\to0$.
In numerical experiments on a sphere and cube with $\betaparam=0$, we see similar convergence to that observed in this section.
This leads us to conjecture that \cref{main_dn_result} will hold for the penalty-free formulation, when $\betaparam=0$.

\subsection{Robin problem}\label{sec:robin_num}
We now consider the case where $\Gamma=\Gamma\R$ and the problem reduces to a Robin problem:
\begin{subequations}
\begin{align}
-\Delta u &=0       &&\text{in } \Omega,\\
\frac{\partial u}{\partial\B\nu} &= \frac{1}{\param} (u - g\D) + g\N
                    &&\text{on }\Gamma,
\end{align}\end{subequations}
for some $\param\in\RR$.

In this section, we compare the method proposed in this paper with the standard method: Find $\lambda\in H^{-1/2}(\Gamma)$ such that
\begin{align}\label{classicalRobin}
\left\langle\bopW u,v\right\rangle+\left\langle\frac1\param\left(\tfrac12\bopI-\bopKadj\right)u,v\right\rangle
&=\left\langle\left(\tfrac12\bopI-\bopKadj\right)\left(\frac1\param g\D+g\N\right),v\right\rangle&&\forall\mu\in H^{-1/2}(\Gamma).
\end{align}

\begin{figure}
\centering
\begin{tikzpicture}
\begin{axis}[height=58mm,width=58mm,
             xlabel={$\betaparam$},xmode=log,
             ylabel={$\param$},ymode=log,
             zlabel={\errorlabel{\R}},zmode=log,
             zmin=0.1,zmax=100000
]
\RthreeD table {img/data/conlin/robin3d_error.dat};
\end{axis}
\end{tikzpicture}
\hfill
\begin{tikzpicture}
\begin{axis}[height=58mm,width=58mm,
             xlabel={$\betaparam$},xmode=log,
             ylabel={$\param$},ymode=log,
             zlabel={Number of GMRES iterations},zmin=0,zmax=200
]
\RthreeD table {img/data/conlin/robin3d_iterations_capped.dat};
\end{axis}
\end{tikzpicture}
\caption{The dependence of the error on $\param$ and $\betaparam$
for the Robin problem on the unit sphere with $h=0.1$, with $k=l+1=1$.
Here we use $\betaparam\D=\betaparam h^{-1}$ and $\betaparam\N=\betaparam h$.
}
\label{fig:con-robin3d}
\end{figure}
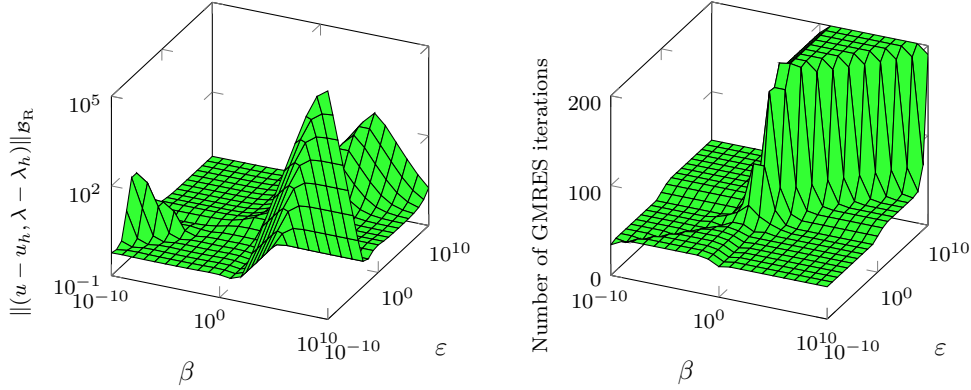

Again, we begin letting $k=l+1=1$.
Here we use 
\begin{equation*}
\betaparam\R:=\frac{\param\betaparam\N+\betaparam\D}{\param+1},
\end{equation*}
where $\betaparam\D=\betaparam h^{-1}$
and $\betaparam\N=\betaparam h$, for some constant $\betaparam$, as in the mixed Dirichlet--Neumann case.

The dependence of the error and iteration count on both $\param$ and $\betaparam$, on a grid with $h=0.1$, is shown in \cref{fig:con-robin3d}.
The convergence as $h$ is reduced
for $\param=\frac1{300}$, $\param=1$, and $\param=300$, and using $\betaparam=0.01$,
is shown in \cref{fig:con-robin1}. In this case, order 1.5 convergence is observed.

\begin{figure}
\centering
\begin{tikzpicture}

\begin{axis}[small,axis equal,axis on top,
             xlabel={$h$},xmode=log,x dir=reverse,
             ylabel near ticks,ylabel={\errorlabel{\R}},ymode=log,
]
\referenceplot table {%
0.4 32.925
0.02 0.3681126907959028
};
\Roneplot table {img/data/conlin/robin1_convergence_DD.dat};
\Rtwoplot table {img/data/conlin/robin1_convergence_DN.dat};
\Rthreeplot table {img/data/conlin/robin1_convergence_NN.dat};
\end{axis}
\end{tikzpicture}
\hfill
\begin{tikzpicture}
\begin{axis}[small,axis on top,
             xlabel={$h$},x dir=reverse,xmode=log,
             ylabel near ticks,ylabel={Number of GMRES iterations},ymin=0,ymax=100
]
\Roneplot table {img/data/conlin/robin1_iterations_DD.dat};
\Rtwoplot table {img/data/conlin/robin1_iterations_DN.dat};
\Rthreeplot table {img/data/conlin/robin1_iterations_NN.dat};
\end{axis}
\end{tikzpicture}

\caption{The convergence (left) and iteration counts (right) of the penalty method for the Robin problem with
$\param=300$ (\Ronedesc),
$\param=1$ (\Rtwodesc) and
$\param=1/300$ (\Rthreedesc)
on the unit sphere, using $k=l+1=1$ and $\betaparam=0.01$.
The dashed line shows order 1.5 convergence.
Here we use $\betaparam\D=\betaparam h^{-1}$ and $\betaparam\N=\betaparam h$.}
\label{fig:con-robin1}
\end{figure}
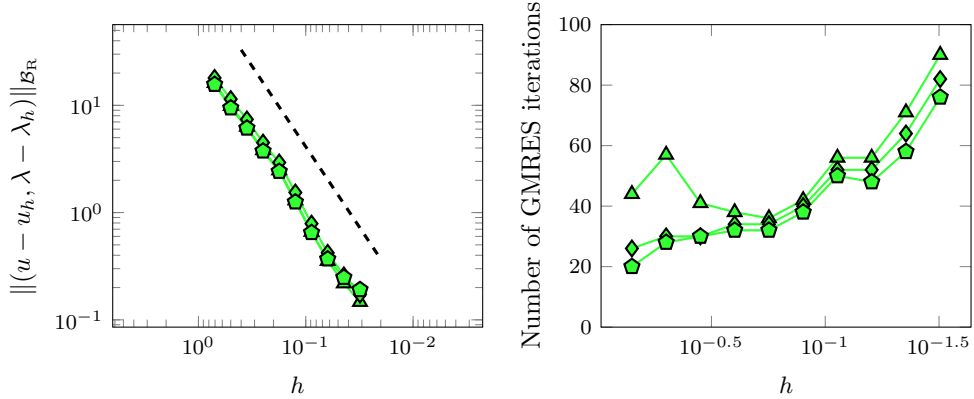

As in the mixed Dirichlet--Neumann case, when $k=l=1$, we may replace the bound on $\betaparam\N$ with $\betaparam\N\lesssim h^{-1}$.
Again, we take $\betaparam\D=\betaparam\N=\betaparam$ for some constant $\betaparam$. The dependence of the error
and iteration count on both $\betaparam$ and $\param$ is shown in \cref{fig:lin-robin3d}. As in the previous case, $\betaparam=0.01$ looks to
be a suitable choice for the parameter.

The convergence as we reduce $h$
for $\param=\frac1{300}$, $\param=1$, and $\param=300$, and using $\betaparam=0.01$,
is shown in \cref{fig:lin-robin1}. In this case, order 2 convergence is observed. For the standard method \cref{classicalRobin}, the
same order of convergence and errors of almost identical size are oberved. For the standard method, the number of iterations required to
solve the system is higher for smaller values of $\param$; for the penalty method, the number of iterations is less affected by the value of
$\param$, leading to lower iteration counts than the standard method for small values of $\param$.

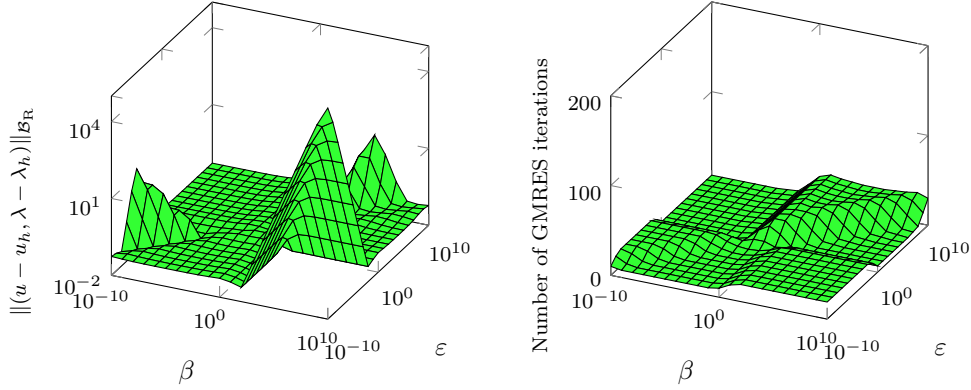
\begin{figure}
\centering
\begin{tikzpicture}
\begin{axis}[height=58mm,width=58mm,
             xlabel={$\betaparam$},xmode=log,
             ylabel={$\param$},ymode=log,
             zlabel={\errorlabel{\R}},zmode=log,
             zmin=0.01,zmax=100000
]
\RthreeD table {img/data/linlin/robin3d_error.dat};
\end{axis}
\end{tikzpicture}
\hfill
\begin{tikzpicture}
\begin{axis}[height=58mm,width=58mm,
             xlabel={$\betaparam$},xmode=log,
             ylabel={$\param$},ymode=log,
             zlabel={Number of GMRES iterations},zmin=0,zmax=200
]
\RthreeD table {img/data/linlin/robin3d_iterations_capped.dat};
\end{axis}
\end{tikzpicture}
\caption{The dependence of the error on $\param$ and $\betaparam$
for the Robin problem on the unit sphere with $h=0.1$, with $k=l=1$.
Here we use $\betaparam\D=\betaparam\N=\betaparam$.
}
\label{fig:lin-robin3d}
\end{figure}

\begin{figure}
\centering
\begin{tikzpicture}

\begin{axis}[small,axis equal,axis on top,
             xlabel={$h$},xmode=log,x dir=reverse,
             ylabel near ticks,ylabel={\errorlabel{\R}},ymode=log,
]
\referenceplot
table {%
0.35355 4
0.02 0.01280024550870886
};
\input{img/robin/robin_c}
\Roneplot table {img/data/linlin/robin1_convergence_DD.dat};
\Rtwoplot table {img/data/linlin/robin1_convergence_DN.dat};
\Rthreeplot table {img/data/linlin/robin1_convergence_NN.dat};
\end{axis}
\end{tikzpicture}
\hfill
\begin{tikzpicture}
\begin{axis}[small,axis on top,
             xlabel={$h$},x dir=reverse,xmode=log,
             ylabel near ticks,ylabel={Number of GMRES iterations},ymin=0,ymax=40
]
\input{img/robin/robin_i}
\Roneplot table {img/data/linlin/robin1_iterations_DD.dat};
\Rtwoplot table {img/data/linlin/robin1_iterations_DN.dat};
\Rthreeplot table {img/data/linlin/robin1_iterations_NN.dat};
\end{axis}
\end{tikzpicture}
\caption{The convergence (left) and iteration counts (right) of the penalty method (\Rcolordesc)
compared to the standard method \cref{classicalRobin} (\usualvariantcolordesc),
for the Robin problem with
$\param=300$ (\onedesc),
$\param=1$ (\twodesc) and
$\param=1/300$ (\threedesc)
on the unit sphere, using $k=l=1$ and $\betaparam=0.01$.
The dashed line shows order 2 convergence.
Here we use $\betaparam\D=\betaparam\N=\betaparam$.
}
\label{fig:lin-robin1}
\end{figure}

Again, we could consider the penalty-free formulation for the Robin problem. However, \cref{fig:con-robin3d,fig:lin-robin3d} suggest that as
$\betaparam\to0$, the error increases for some values of $\param$. This increased error can also be observed in the numerical
experiments we have run with $\betaparam=0$. Hence in the Robin case, the penalty term is necessary and \cref{main_robin_result} does not hold
for $\betaparam\R=0$.


\section{Conclusions} We have analysed and demonstrated the effectiveness of Nitsche type coupling methods for 
boundary element formulations. In particular, for Robin and mixed Neumann/Dirichlet boundary conditions these 
are simpler than the strong imposition of boundary conditions since the boundary condition only enters the 
equations through a sparse operator.

An open problem is preconditioning. While the iteration counts in the presented examples were already 
practically useful, for large and complex structures preconditioning is still essential. The hope is to use 
the properties of the Calder\'on projector to build effective operator preconditioning techniques for the 
presented Nitsche type frameworks.

An extension of the presented method to FEM/BEM formulations is currently in preparation. Other directions 
are the Helmholtz and Maxwell problems. Although the analysis for these cases is more involved, we expect that 
their implementation will be structurally similar to the presented Laplace case.

\textbf{Acknowledgements.} Erik Burman and Timo Betcke were supported by EPSRC
grants EP/P01576X/1 and EP/P012434/1.
\bibliographystyle{siamplain}
\bibliography{references}
\end{document}